\documentclass[11pt,reqno]{article}
\usepackage[dvipsnames]{xcolor}
\usepackage{multirow}
\usepackage{upgreek}
\usepackage{tikz}
\usepackage[english]{babel}
\usepackage{array}
\usepackage{comment}
\usepackage{graphicx}
\usepackage[dvipsnames]{xcolor}
\usepackage[centertags]{amsmath}
\usepackage{amsfonts,amssymb,amsthm}
\usepackage[]{hyperref}
\usepackage[latin1]{inputenc}
\usepackage{xcolor} 
\usepackage[noadjust]{cite}
\usepackage{mathtools,xparse}
\usepackage{mathtools}
\usepackage{bbm}
\usepackage{geometry}

\numberwithin{equation}{section}

\newtheorem{theorem}{Theorem}[section]
\newtheorem{lemma}[theorem]{Lemma}

\newtheorem{corollary}[theorem]{Corollary}
\newtheorem{remark}[theorem]{Remark}
\newtheorem{definition}[theorem]{Definition}

\newcommand{\T}{\mathbb{T}}
\newcommand{\R}{\mathbb{R}}
\newcommand{\C}{\mathbb{C}}
\newcommand{\norm}[1]{\left\lVert#1\right\rVert}
\newcommand{\abs}[1]{\left\vert#1\right\vert}
\newcommand{\inner}[2]{\left\langle #1, #2 \right\rangle}
\newcommand{\eps}{\varepsilon}
\newcommand{\La}{\Lambda}

\title{Global Well-Posedness and Uniform Long-Time Control for a Class of Periodic NLS-Dysthe Equations}
\author{Andrea Belloni\footnote{Dipartimento di Matematica, Politecnico di Milano, Piazza Leonardo Da Vinci 32, 20133, Milano, Italy email\,\,:\,\,\texttt{andrea.belloni@polimi.it, filippo.giuliani@polimi.it}}, Filippo Giuliani$^*$.}
\date{}

\begin{document}

\maketitle
\begin{abstract}

We study a class of periodic NLS-Dysthe equations of the form
\[
i\partial_t u+a\partial_x^2u+b|u|^2u
+i\eps\left(c\partial_x^3u+d|u|^2\partial_xu
+eu^2\partial_x\overline u\right)=0,
\qquad x\in\mathbb T,
\]
where $a,b,c,d,e\in\mathbb R$, $abc\neq0$, and $ \eps >0$. For $0<\eps\ll1$, equations of this type arise as higher-order modulation models for some nonlinear dispersive PDEs, including the Klein-Gordon equation and the Serre-Green-Naghdi equation modeling the evolution of nonlinear shallow water waves.

We prove local well-posedness in $H^s(\mathbb T)$ for $s>3/2$ for arbitrary admissible coefficients, and establish an $H^1$ continuation criterion for $s\ge2$. Under the additional assumption $e\neq0$, a lower-order conservation law provides uniform $H^1$ control and, together with the continuation criterion, yields global well-posedness for $s\ge2$. Unlike the previous $H^1$ global theory based on Takaoka's periodic well-posedness estimates \cite{takaoka2000well} our result does not require any non-resonance condition on the coefficients.

Our approach relies on the construction of modified higher-order energies, following the strategy developed in \cite{PTV23, hayashi2025global}.

Motivated by higher-order modulation theory, we exploit the conservation laws of cubic NLS and their invariance under the commuting complex mKdV flow to establish Sobolev bounds uniform in the modulation parameter $\varepsilon$ on time intervals of order $\varepsilon^{-1}$. This corresponds to physical times of order $\varepsilon^{-3}$, which is the natural time scale anticipated for the NLS-Dysthe approximation. 
\end{abstract}
\tableofcontents

\section{Introduction}

Modulation theory focuses on the study of the slow evolution of nearly monochromatic wavetrains in nonlinear dispersive systems. The first formulations in this field are Whitham's averaged-Lagrangian theory in \cite{whitham1965lagrangian} and the NLS-modulation theory for deep-water gravity waves derived by Zakharov in \cite{zakharov1968stability}. In its standard form, the modulation ansatz has the schematic structure
\begin{equation}\label{ansatz}
U(t,x)\sim \eps A\big(\eps(x-c_g t),\eps^2t\big)e^{i(k_0x-\omega_0t)}+\mathrm{c.c.} \, \, , \qquad 0< \eps \ll 1,
\end{equation}
where $k_0$ is the carrier wavenumber, $\omega_0=\omega(k_0)$ is the corresponding frequency determined by the linear dispersion relation of the underlying system and $c_g=\omega'(k_0)$ is the associated group velocity. 
Substituting the ansatz \eqref{ansatz} into the underlying dispersive equations and matching powers of $\varepsilon$ leads to a cubic nonlinear Schr\"odinger equation governing the slow evolution of the envelope $A$. The resulting wave-packet approximation is expected to describe the physical dynamics on time scales of order $\varepsilon^{-2}$. 
The rigorous justification of this approximation has been extensively investigated for semilinear and quasilinear dispersive systems, including Klein-Gordon equations and gravity water waves; see, among others, 
\cite{craig1992modulation,schneider1998kdv,schneider1998hyperbolic,dullschneider2006boussinesq,schneider2011miura,schneiderwayne2011,totzwu2012,chirilusbruckner2015nls,totz2015three,dullschneiderwayne2016,dull2017kleingordon,ifrimtataru2019,dull2021waterwaves}
 and the references therein.

Higher-order modulation equations incorporate corrections to cubic NLS and are expected to provide a more accurate description of wave-packet dynamics. 
Their relevance
has been investigated numerically
\cite{lomei1985,trulsendysthe1996,craigguyennesulem2021}, particularly in connection
with the description of extreme and rogue waves
\cite{dysthekrogstadmuller2008,farazmandsapsis2017,dematteis2018}.\\
The derivation of higher-order NLS modulation equations was pioneered by Dysthe \cite{dysthe1979} in the context of deep-water gravity waves. Related models have subsequently been derived in different physical settings, including nonlinear wave equations and shallow-water models \cite{craigguyennesulem2010}, \cite{congy2019}. 
Throughout this paper, we refer to this family of higher-order modulation models as NLS-Dysthe equations.

\smallskip

Normal-form transformations play an important role in the rigorous justification of NLS approximations on the line, where they are used to eliminate non-resonant interactions; see, for instance, \cite{schneiderwayne2011,dullschneiderwayne2016,ifrimtataru2019}. Extending these approximations beyond the natural NLS time scale is expected to require successive normal-form reductions, a challenging task in the presence of a continuous spectrum and nontrivial resonance manifolds. This motivates the study of periodic wave dynamics, where the discrete spectrum offers a potentially more tractable framework for implementing such a strategy. However, discreteness does not eliminate resonance-related difficulties: exact resonances may obstruct the validity of the NLS approximation \cite{SCHNEIDER2005,SSZ2015}, while near-resonant interactions can give rise to small divisors.

In this perspective, we consider the underlying physical model on the large torus $\mathbb{T}_{\varepsilon^{-1}}:=\mathbb{R}/2\pi\varepsilon^{-1}\mathbb{Z}$, $0<\eps\ll 1$. The large spatial period allows one to retain some of the dispersive effects characteristic of the Euclidean setting on suitable time scales, while the frequency lattice becomes increasingly dense as $\varepsilon\to0$, leading to an accumulation of near-resonant interactions. Under the modulation ansatz \eqref{ansatz}, the introduction of the slow spatial variable leads to an envelope equation on the fixed torus $\mathbb{T}: = \R / 2\pi\mathbb{Z}$.

\medskip

In the present paper we consider the following class of periodic NLS-Dysthe equations
\begin{equation}\label{app:D}
\begin{cases}
i\partial_t u + a\partial_x^2 u + b\abs{u}^2 u
+ i\eps\big(c\partial_x^3 u + d\abs{u}^2\partial_x u + e u^2\partial_x\overline{u}\big)=0,\\
u|_{t=0}=u_0,
\end{cases}
\end{equation}
where $x\in\T$, $t\in\R$, $u_0\in H^s(\T;\C)$, and $a,b,c,d,e\in\R$ are independent of $\eps > 0 $. Throughout the paper we assume $abc\neq0$.

The classical Dysthe equation for deep-water gravity waves \cite{dysthe1979} contains an additional nonlocal mean-flow contribution and thus falls outside the class \eqref{app:D}. However, such a nonlocal term is not a general feature of higher-order modulation equations arising in fluid dynamics. Indeed, an example within the class \eqref{app:D} is provided by the NLS-Dysthe equation coming from the Serre-Green-Naghdi system
\[
\begin{cases}
h_t+(hu)_x=0,\\[2mm]
u_t+u u_x+h_x
=\displaystyle \frac{1}{3h}\partial_x\!\left[h^3\big(u_{xt}+u u_{xx}-u_x^2\big)\right],
\end{cases}
\]
where $x\in\T$, $t\in\R$, $h$ is the total fluid depth and $u$ is the depth-averaged horizontal velocity. 
This is a fully nonlinear, weakly dispersive model for shallow-water waves, whose approximation of the water-wave equations has been rigorously justified, see, for instance, \cite{li2006shallow,lannes2013}. Its higher-order modulation equation has been derived in \cite{congy2019}.

The class \eqref{app:D} also includes higher-order modulation equations arising from the real cubic Klein-Gordon equation. In \cite{craigguyennesulem2010}, a Hamiltonian modulation approach yields an equation with $e=0$ in suitable canonical envelope variables. By contrast, a direct multiple-scales expansion in physical envelope variables leads to an equation of the same form with $e\neq0$ for nonzero carrier wavenumber. This illustrates the non-universality of higher-order modulation equations, whose explicit form depends on the choice of the wave-packet ansatz and envelope variables.

\medskip

The present paper establishes local and global well-posedness results for the periodic NLS-Dysthe equations \eqref{app:D}, together with Sobolev bounds uniform in $\varepsilon$ on time intervals of order $O(\varepsilon^{-1})$. In terms of the original physical variables, where the slow time is given by $T=\varepsilon^2 t$, this corresponds to times of order $O(\varepsilon^{-3})$, the natural time scale anticipated for the higher-order Dysthe approximation, beyond the standard NLS time scale $O(\varepsilon^{-2})$. However, the rigorous justification of periodic NLS-Dysthe approximations on this longer time scale remains an open problem. A necessary step toward such a justification is to establish uniform control of the solutions to the modulation equation over the corresponding time intervals. Our results provide this control without any smallness assumption on the initial envelope.

\medskip

For the well-posedness analysis, it is useful to decouple the higher-order corrections to the cubic NLS from the small modulation parameter. We therefore first study the more general Cauchy problem
\begin{equation}\label{eq:main_pde}
\begin{cases}
i\partial_t u + a\partial_x^2 u + b\abs{u}^2 u
+ i c\partial_x^3 u + i d\abs{u}^2 \partial_x u
+ i e u^2 \partial_x \overline{u} = 0,\\
u|_{t=0} = u_0.
\end{cases}
\end{equation}
The point of considering \eqref{eq:main_pde} is that the local and global arguments do not rely on the asymptotic size of the higher-order coefficients. Before stating our results, we discuss some previous works on this general model. Outside modulation theory, the model \eqref{eq:main_pde} also appears in the description of the nonlinear propagation of pulses in optical fibers (see for instance \cite{kodama1987nonlinear,kodama1985optical}). In this field the equation is usually called a higher-order nonlinear Schr\"odinger equation or Airy--Schr\"odinger equation. The Cauchy problem on the line has been studied extensively. Laurey \cite{laurey1997cauchy} proved local well-posedness at high regularity and global well-posedness in $H^s(\R; \C)$ for $s\geq1$ using conservation laws. Carvajal \cite{carvajal2006sharp} later obtained sharp global well-posedness in $H^s(\R; \C)$, $s>1/4$, under a structural relation among the coefficients, using almost conserved energies and the $I$-method.
In the periodic setting, Takaoka \cite{takaoka2000well} established local well-posedness in $H^s(\mathbb T;\mathbb C)$ for $s\geq 1/2$ under the non-resonance condition
\[
\frac{2a}{3c}\notin\mathbb Z,
\]
which excludes certain resonances of the linear dispersion and plays a crucial role in the underlying Fourier restriction estimates. In contrast, our approach does not impose any arithmetic condition on the coefficients, at the expense of a higher regularity threshold: we obtain local well-posedness for $s>3/2$ and global and long-time results for $s\geq 2$. This level of regularity is consistent with the requirements of rigorous modulation theory, where NLS approximations are typically justified for sufficiently smooth envelope profiles; for instance, the results in \cite{schneiderwayne2011,ifrimtataru2019} assume $H^6$ and $H^3$ regularity, respectively.

The absence of non-resonance assumptions is particularly relevant for the study of \eqref{app:D}. Indeed, Takaoka's condition would become
\[
\frac{2a}{3c\varepsilon}\notin\mathbb Z.
\]
As $\varepsilon$ tends to zero, resonant values accumulate at zero, and arbitrarily small non-resonant values can lie arbitrarily close to resonance. Consequently, estimates requiring a uniform positive lower bound on $\operatorname{dist}\left(\frac{2a}{3c\eps},\mathbb{Z}\right)$ cannot hold throughout the perturbative regime. Our approach avoids this obstruction, allowing us to establish well-posedness and long-time bounds without excluding any values of $\varepsilon>0$.

Our first result is the following.
\begin{theorem}[$H^1$ Continuation criterion]\label{GWPG}
Let $abc\neq 0$, $s>3/2$ and $u_0 \in H^s(\T;\C)$. Then there exist $T>0$ and a unique solution $u \in C([-T,T],H^s(\T,\C))$ to \eqref{eq:main_pde}, and the flow depends continuously on the initial datum. Moreover, if $s\geq2$ and
\[
u\in C((-T_{1,\textnormal{max}},T_{2,\textnormal{max}}),H^s(\T;\C))
\]
is the maximal solution, for some $T_{i,\textnormal{max}}>0$, then the following alternative holds:
\begin{itemize}
\item $(-T_{1,\textnormal{max}},T_{2,\textnormal{max}})=\R$;
\item if $T_{2,\textnormal{max}}<+\infty$, then
\[
\limsup_{t\nearrow T_{2,\textnormal{max}}}\|u(t)\|_{H^1(\T;\C)}=+\infty;
\]
\item if $T_{1,\textnormal{max}}<+\infty$, then
\[
\limsup_{t\searrow -T_{1,\textnormal{max}}}\|u(t)\|_{H^1(\T;\C)}=+\infty.
\]
\end{itemize}
\end{theorem}

The local theory is obtained by a regularization argument and Sobolev estimates. The main point in the continuation argument is instead the construction of \emph{modified higher-order energies} which remove the derivative loss in the standard energy estimate. Our implementation is close in spirit to the modified-energy method used for periodic derivative NLS equations \cite{hayashi2025global} and generalized KdV equations \cite{PTV23}; see also the discussion in \cite{visciglia2021modified}.

When $e\neq0$, the $H^1$ continuation criterion can be closed by a lower-order conservation law. More precisely, for smooth solutions we have the conserved quantity
\[
I(u)=\int_{\T}|u_x|^2\,dx
-\frac{d+e}{6c}\int_{\T}|u|^4\,dx
+\frac{3bc-a(d+e)}{3ce}\,\Im\int_{\T}\overline{u}\,u_x\,dx.
\]
Its leading term is $\|u_x\|_{L^2}^2$, while the remaining terms are controlled, at fixed mass, by lower-order quantities. This conservation law is already present, up to notation and normalization, in the real-line theory of Laurey \cite{laurey1997cauchy}; see also \cite{carvajal2006sharp}, where it is recorded explicitly for the same class of higher-order Schr\"odinger equations. The same algebraic identity holds for smooth periodic solutions, since the integrations by parts produce no boundary terms. 
 Consequently, conservation of $I(u)$ and of the $L^2$ norm yields a uniform-in-time $H^1$ bound, which, together with the continuation criterion of Theorem \ref{GWPG}, establishes global well-posedness for $e\neq 0$.

\begin{corollary}[$H^2$ Global well-posedness]\label{e0GWP}
Let $s\geq 2$ and $u_0\in H^s(\T;\C)$. If $abc\neq 0, e\neq0$, then the Cauchy problem \eqref{eq:main_pde} is globally well posed.
\end{corollary}

\begin{remark}
The assumption $e\neq 0$ is naturally satisfied by higher-order modulation equations arising from relevant dispersive models. In particular, it occurs in the modulation equation derived from the Serre-Green-Naghdi system \cite{congy2019} for generic carrier wavenumbers, as well as in the real cubic Klein-Gordon equation, when the modulation expansion is performed directly in the physical envelope variables, for nonzero carrier wavenumber.
\end{remark}

\begin{remark}
The argument used to prove Corollary \ref{e0GWP} can be adapted to the case $e=0$ and $d=3bc/a$, see Remark \ref{rem:Hirota}. This relation gives the integrable Hirota equation introduced in \cite{hirota1973exact}.
\end{remark}

\begin{remark} We mention the work of Scialom and Bragan\c{c}a \cite{scialombraganca2019periodic}, who established periodic global well-posedness in $H^1$ for a coupled system of third-order nonlinear Schr\"odinger equations under suitable structural assumptions. Their system admits \eqref{eq:main_pde} as a scalar reduction, but their global result applies only to the case $e=0$. Moreover, their local well-posedness theory relies on a non-resonance condition on the linear dispersion. 
\end{remark}

We now turn to the perturbative problem \eqref{app:D}. Our approach exploits the higher conservation laws of cubic NLS, which are also preserved by the commuting complex mKdV flow. This allows us to isolate the non-integrable contribution to the evolution of these energies, leading to the following uniform long-time control result.

\begin{theorem}[Uniform long-time control]\label{thm:1}
Let $s\geq2$ and $[s]:=\max\{k\in\mathbb N\,:\,k\leq s\}.$ For any $u_0\in H^s(\T;\C)$, there exists $\tilde K_0>0$ such that the following holds. For any $\tilde K_1>\tilde K_0$, there exists $\tilde K_2>0$ such that, for every $\eps\in(0,1)$, if $u\in C((-T_{1,\textnormal{max}},T_{2,\textnormal{max}}); H^s(\T;\C))$ is the maximal solution of \eqref{app:D}, then
\begin{equation}\label{stima:longtime}
T_{1,\textnormal{max}},T_{2,\textnormal{max}} > \tilde T:=\tilde K_2\eps^{-1}.
\end{equation}
Moreover,
\begin{equation}\label{stability}
\sup_{t\in[-\tilde T,\tilde T]} \|u(t)\|_{H^{[s]}} \leq\tilde K_1.
\end{equation}
The constants $\tilde K_1$ and $\tilde K_2$ are independent of $\eps$.
\end{theorem}

For integer $s\geq 2$, Theorem \ref{thm:1} provides uniform bounds in the same Sobolev space as the initial datum. For non-integer $s>2$, uniform control is obtained in $H^{[s]}$, with a loss of less than one derivative, reflecting the integer-indexed structure of the NLS conservation laws. Such a loss may still be compatible with rigorous modulation arguments, where derivative losses can also arise in the estimates for the approximation error; see, for instance, \cite{schneiderwayne2011,totzwu2012,dullschneiderwayne2016,ifrimtataru2019}. Importantly, these bounds hold for arbitrary initial data independent of $\varepsilon$, without any smallness assumption.
\begin{remark}
In the preprint \cite{kishimotouniform}, Kishimoto established uniform local well-posedness and convergence to cubic NLS as $\varepsilon\to0$ under a non-resonance assumption. Unlike that result, which concerns fixed time intervals, our theorem provides control on intervals of order $O(\varepsilon^{-1})$, without imposing any non-resonance condition.
\end{remark}

We conclude the introduction by comparing our results with the existing well-posedness theory for the nonlocal Dysthe equation arising from deep-water gravity waves. The Dysthe model obtained in this context does not belong to the family \eqref{eq:main_pde}, since it contains an additional nonlocal mean-flow term. In one space dimension, and up to normalization, this contribution has the form $u|D|(|u|^2)$, while in two dimensions the corresponding term involves a nonlocal Fourier multiplier. This changes the analytic structure substantially, and in particular removes the local estimate used to prove our result. In the Euclidean two-dimensional setting, Grande, Kurianski and Staffilani proved local well-posedness in $H^s(\R^2)$ for $s>1$ and ill-posedness below $L^2$ in the sense of lack of $C^3$ regularity of the flow map \cite{grandekurianskistaffilani2021}. Mosincat, Pilod and Saut subsequently pushed the local theory to $H^s(\R^2)$ for every $s>0$ and proved small-data global well-posedness and scattering in the critical space $L^2(\R^2)$ \cite{mosincatpilodsaut2021}. Their analysis makes essential use of the dispersive structure available on $\R^2$, through linear and bilinear Strichartz estimates and Fourier restriction techniques.

The periodic problem presents further challenges, as the dispersive estimates available on $\mathbb R^2$ cannot be directly transferred to the torus. In this direction, Heller and Shao \cite{hellershao2021} established periodic Strichartz and multilinear estimates adapted to the Dysthe dispersion, together with ill-posedness results at low Sobolev regularity. Developing a well-posedness theory for the periodic nonlocal Dysthe equation, possibly extending the energy-based approach of the present paper, remains a natural direction for future research.

\medskip

\noindent\textbf{Plan of the paper.} The paper is organized as follows. In Section \ref{LWP}, we prove local well-posedness for the generalized Cauchy problem \eqref{eq:main_pde} by approximation and Sobolev energy estimates. In Section \ref{MDE}, we construct modified energies and prove the $H^1$ blow-up criterion in Theorem \ref{GWPG}; the conserved quantity above then yields Corollary \ref{e0GWP}. Finally, in Section \ref{per:DNLS}, we study the perturbative model \eqref{app:D} and prove the uniform estimate and the $O(\eps^{-1})$ lifespan of Theorem \ref{thm:1}.

\textbf{Notation:} for any $x,y\in\R$ we write $x\lesssim y$ if there exists $C>0$ such that $x\leq Cy$. If also $y \lesssim x$, we denote $ x \sim y $. In the following sections, $x$ and $y$ may denote functions depending on various parameters. Whenever some parameters appear explicitly in r.h.s. of the inequality, the implicit constant $C$ is understood to depend on the remaining, omitted parameters.

\bigskip

	\noindent
	{\bf Acknowledgements:} A. Belloni and F. Giuliani have been supported by INdAM/GNAMPA.

F. Giuliani has been supported by the INdAM project "Growth of Sobolev norms and energy cascades in dispersive Hamiltonian PDE" CUP E53C25002010001.

\vspace{0.2cm}

\noindent\textbf{AI declaration:}The authors retain full responsibility for the mathematical ideas, accuracy, originality, integrity, and content of the manuscript.

During the preparation and revision of this manuscript, ChatGPT 5.6 (OpenAI) was used to assist with language editing, literature discovery, and consistency checks with all the outputs independently verified by the authors.

\noindent\textbf{Statements and Declarations.} Data sharing is not applicable to this article as no datasets were generated or analyzed during the current study.

\noindent\textbf{Competing Interests}: The authors have no conflicts of interest to declare.

\section{Local Existence}\label{LWP}

\subsection{Functional setting}

{
We establish some standard functional frameworks and notations used throughout this paper. We denote by $L^2(\T; \C)$ the space of square-integrable complex-valued functions on the torus $\T$, identifying functions $f \in L^2(\T; \C)$ with their Fourier sequences $\{\hat{f}(k)\}_{k \in \mathbb{Z}}$. The Fourier coefficients are defined as
\[
\hat{f}(k) := \frac{1}{\sqrt{2\pi}}\int_{\T} f(x) e^{-ikx} \, dx,
\]
and we equip the space with the standard inner product
\[
\langle f, g \rangle_{L^2} := \int_{\T} f(x) \overline{g}(x) \, dx,
\]
which induces the norm 
\[
\|f\|_{L^2}^2 := \sum_{k \in \mathbb{Z}} |\hat{f}(k)|^2 < \infty.
\]
For any real number $s > 0$, we define the fractional Sobolev space $H^s(\T; \C) \subset L^2(\T; \C)$ as the space of functions for which the following norm is finite
\[
\|f\|_{H^s}^2 := \sum_{k \in \mathbb{Z}} (1+|k|^2)^s |\hat{f}(k)|^2 < \infty.
\]
We identify $H^0(\T;\C) = L^2(\T;\C)$. Associated with this space, we introduce the Bessel potential operator $\Lambda^s := (1-\partial_x^2)^{s/2}$, which acts as a Fourier multiplier satisfying
\[
 \widehat{\Lambda^s f}(k) = (1+|k|^2)^{s/2} \hat{f}(k).
\]
Consequently, $\|f\|_{H^s} = \|\Lambda^s f\|_{L^2}$. We also define the space of smooth functions $H^{\infty}(\T;\C)$ as the intersection of all integer-order Sobolev spaces
\[
H^{\infty}(\T;\C) := \bigcap_{m \in \mathbb{N}} H^m(\T;\C).
\]
Furthermore, to rigorously manage the nonlinear terms and commutators in the subsequent estimates, we will systematically rely on the Lebesgue spaces $L^p(\T; \C)$ and the standard Sobolev spaces $W^{k,\infty}(\T; \C)$. In particular, we recall the definition of the $W^{1,\infty}$ norm, given by
\[
\|f\|_{W^{1,\infty}} := \|f\|_{L^\infty} + \|\partial_x f\|_{L^\infty}.
\]
Finally, in the next sections, we will usually drop the domain and codomain from the notation of the functional spaces (e.g., writing $L^2$ or $H^s$ in place of $L^2(\T; \C)$ or $H^s(\T; \C)$) for the sake of brevity.
}

\subsection{Approximate problem}
The study of local existence follows a standard approach based on the use of frequency cut-offs (see for instance \cite{ambrose2015local}, \cite{hayashi2025global}). 
We first prove local existence and uniform bounds for approximate equations of the full problem \eqref{eq:main_pde}.
 From these uniform bounds, we subsequently deduce local existence for the original equation through a rigorous limiting process.

We introduce the spectral projection operator $J_N$. For any cut-off frequency $N \in [1, +\infty)$, its action on a function $f \in L^2$ is defined via its Fourier series as
\begin{equation} \label{eq:JN_def}
(J_N f)(x) = \sum_{\abs{k} \le N} \hat{f}(k) \frac{e^{ikx}}{\sqrt{2\pi}}, \quad x \in \T.
\end{equation}
We summarize the main properties of the operator $J_N$ in the following lemma.

\begin{lemma}[Lemma 2.1 in \cite{hayashi2025global}] \label{lem:JN_properties}
For $N \in [1, +\infty)$, $s \geq 0$ and for any functions $f, g \in L^2$, the following properties hold:
\begin{enumerate}
\item $J_N^2 = J_N,$ \item $\inner{J_N f}{g}_{L^2} = \inner{f}{J_N g}_{L^2}$, \item $\norm{J_N f}_{L^2} \le \norm{f}_{L^2}, \qquad \norm{J_N f}_{H^s} \lesssim N^s\norm{f}_{L^2},$ \item $\norm{J_N f}_{H^s} \le \norm{f}_{H^s}, \qquad \norm{(I-J_N) f}_{L^2} \lesssim \frac{1}{N^s} \norm{f}_{H^s}, \quad \forall f \in H^s$, \item $\lim_{N \to \infty} \norm{J_N f - f}_{H^s} =0, \quad \forall f \in H^s$. 
\end{enumerate}
\end{lemma}

 Let us define the linear operator $A$ and the nonlinearity $g(u)$ as
\begin{equation} \label{eq:operators_def}
A = i a\partial_x^2 - c\partial_x^3, \quad g(u) = i b\abs{u}^2 u - d\abs{u}^2 \partial_x u - e u^2 \partial_x \overline{u}.
\end{equation}
We consider the approximate Cauchy problem associated to the NLS-Dysthe equation \eqref{eq:main_pde}
\begin{equation} \label{eq:approx_problem}
\begin{cases}
\partial_t u_N = A u_N + J_N g(J_N u_N), \\
u_N(0) = J_N u_0.
\end{cases}
\end{equation}
We are going to study its solutions following the ideas in \cite{cazenave2003semilinear} (Theorem 3.3.1) and \cite{hayashi2025global} (Lemma 2.2).

\begin{lemma} \label{lem:global_approx}
Let $a,b,c,d,e \in \R$, $abc \neq 0$ and fix $N \ge 1$. For any initial datum $u_0 \in L^2$, there exists a unique global solution $u_N \in C(\R, L^2) \cap C^1(\R, H^\infty)$ of the approximate problem \eqref{eq:approx_problem}. Moreover, the mass conservation holds
\begin{equation}\label{conservaL2}
\norm{u_N(t)}_{L^2}^2 = \norm{J_N u_0}_{L^2}^2 \quad \forall t \in \R.
\end{equation}
\end{lemma}

\begin{proof}
We begin by proving local existence. The problem \eqref{eq:approx_problem} can be rewritten in its equivalent integral form using Duhamel's formula
\begin{equation*}
u_N(t) = U(t)J_N u_0 + \int_0^t U(t-s) J_N g(J_N u_N(s))  ds,
\end{equation*}
where $U(t) = e^{tA}$ is the unitary group generated by the linear operator.
Therefore, defining the integral mapping
\begin{equation}\label{Dua}
\Phi(u)(t) = U(t)J_N u_0 + \int_0^t U(t-s) F_N(u(s)) ds,
\end{equation}
we seek a fixed point $\Phi(u) = u$, with $u \in C([-T,T], L^2)$, for some $T>0$. To apply the Banach fixed-point theorem, we need to verify that there exists $T>0$ such that the integral mapping associated with Duhamel's formula maps $C([-T, T]; L^2)$ into itself and it is a contraction. We start by proving that the nonlinear map $F_N(u) = J_N g(J_N u)$ is locally Lipschitz continuous on $L^2$.

Let $u, v \in L^2$ and denote $\tilde{u} = J_N u$ and $\tilde{v} = J_N v$. Then, exploiting the property of the projection operator (Lemma \ref{lem:JN_properties}), we have
\begin{equation} \label{eq:lipschitz_1}
\norm{F_N(u) - F_N(v)}_{L^2} = \norm{J_N(g(\tilde{u}) - g(\tilde{v}))}_{L^2} \le \norm{g(\tilde{u}) - g(\tilde{v})}_{L^2}.
\end{equation}
Based on the definition of $g$ in \eqref{eq:operators_def}, to estimate the right-hand side of \eqref{eq:lipschitz_1} we consider the various contributions separately. For the cubic term without derivatives, Cauchy-Schwarz gives $\|\tilde u\|_{L^\infty} \lesssim N^{1/2}\|u\|_{L^2}$ and then
\begin{equation}\label{aux.aux}
\norm{\abs{\tilde{u}}^2 \tilde{u} - \abs{\tilde{v}}^2 \tilde{v}}_{L^2} \lesssim (\norm{\tilde{u}}_{L^\infty}^2 + \norm{\tilde{v}}_{L^\infty}^2) \norm{\tilde{u} - \tilde{v}}_{L^2} \lesssim N (\norm{u}_{L^2}^2 + \norm{v}_{L^2}^2) \norm{u - v}_{L^2},
\end{equation}
where we applied the third property of Lemma \ref{lem:JN_properties}.

For the terms involving spatial derivatives, we proceed in a similar fashion. For instance we consider
\begin{equation}\label{auxaux:2}
\norm{\abs{\tilde{u}}^2 \partial_x \tilde{u} - \abs{\tilde{v}}^2 \partial_x \tilde{v}}_{L^2} \le \norm{\abs{\tilde{u}}^2 \partial_x(\tilde{u} - \tilde{v})}_{L^2} + \norm{(\abs{\tilde{u}}^2 - \abs{\tilde{v}}^2) \partial_x \tilde{v}}_{L^2}.
\end{equation}
The first term in this decomposition can be bounded by
\begin{equation*}
\norm{\abs{\tilde{u}}^2 \partial_x(\tilde{u} - \tilde{v})}_{L^2} \le \norm{\tilde{u}}_{L^\infty}^2 \norm{\partial_x(\tilde{u} - \tilde{v})}_{L^2} \lesssim N^2 \norm{\tilde{u}}_{L^2}^2 \norm{u - v}_{L^2},
\end{equation*}
by the same argument used in \eqref{aux.aux} combined with the third property of Lemma \ref{lem:JN_properties}. The second term in the r.h.s. of \eqref{auxaux:2} is estimated again by Cauchy-Schwarz, which gives $\|\partial_x \tilde{v}\|_{L^\infty} \lesssim N^{3/2} \|v\|_{L^2}$.

{Generalizing these straightforward estimates to all terms present in $g$, for every $R>0$ we obtain the Lipschitz bound
\begin{equation} \label{eq:K_lipschitz}
\norm{g(\tilde{u}) - g(\tilde{v})}_{L^2} \le K \norm{u - v}_{L^2},
\end{equation}
for all $u,v\in L^2$ such that $\|u\|_{L^2},\|v\|_{L^2}\le R$, where $K=K(N,R)>0$ and $K(N,R)\to+\infty$ as $N\to\infty$ for fixed $R$.
Fix now $R:=2\|u_0\|_{L^2}$ (the case $u_0=0$ being trivial), and consider the closed ball
\[
B_R:=\left\{u\in C([-T,T],L^2):\|u\|_{C([-T,T],L^2)}\le R\right\}.
\]
Since $g(0)=0$, using \eqref{eq:K_lipschitz}, the $L^2$-boundedness of $J_N$, and the fact that $U(t)$ acts as an isometry on $L^2$, for every $u\in B_R$ we have
\[
\|\Phi(u)\|_{C([-T,T],L^2)}
\le \|u_0\|_{L^2}+TKR
=\frac{R}{2}+TKR.
\]
Thus, choosing $T=T(N,R)>0$ so that $TK\le 1/2$, we obtain $\Phi(B_R)\subseteq B_R$.
Moreover, for every $u,v\in B_R$,
\begin{equation*}
\norm{\Phi(u)-\Phi(v)}_{C([-T,T],L^2)}
\le TK\norm{u-v}_{C([-T,T],L^2)}
\le \frac12\norm{u-v}_{C([-T,T],L^2)}.
\end{equation*}
Hence $\Phi$ is a contraction on $B_R$, and the Banach fixed-point theorem yields a unique local solution. Notice that the local existence time depends only on $N$ and $\|u_0\|_{L^2}$ and, with the above choice, tends to zero as $N\to\infty$.
}

Due to the frequency cutoff operator $J_N$, the Fourier support of the solution is compact, which naturally implies $u_N \in C([-T, T], H^\infty)$. Furthermore, recalling the PDE \eqref{eq:approx_problem}, we have $(I-J_N)u_N$ satisfies the linear problem with initial data zero, which gives $u_N(t) = J_N u_N(t)$. It follows  that $\partial_t u_N \in C([-T, T], H^\infty)$, yielding $u_N \in C^1([-T, T], H^\infty)$.

To upgrade this result and establish global existence for the approximate problem (hence, with fixed $N$), we demonstrate the conservation of the $L^2$ mass. Since the local existence time $T$ depends only on the fixed parameter $N$ and the $L^2$-norm of the initial datum, conservation of the $L^2$-norm allows us to iterate the local existence result. 

Computing the time derivative of the $L^2$ norm gives
\begin{equation*}
\frac{1}{2} \frac{d}{dt} \norm{u_N}_{L^2}^2 = \Re \inner{\partial_t u_N}{u_N}_{L^2} = \Re \inner{A u_N}{u_N}_{L^2} + \Re \inner{J_N g(J_N u_N)}{u_N}_{L^2}.
\end{equation*}
Since $A$ is skew-adjoint, the first term vanishes. For the second term, recalling the notation $\tilde{u} = J_N u_N=u_N$ and exploiting the self-adjointness of $J_N$, we are left to prove that
\begin{equation}\label{aggiunta}
\Re \int_{\T} (i b\abs{\tilde{u}}^2 \tilde{u} - d\abs{\tilde{u}}^2 \partial_x \tilde{u} - e \tilde{u}^2 \partial_x \overline{\tilde{u}}) \overline{\tilde{u}} dx = 0.
\end{equation}
The first term vanishes because the integrand $i b\abs{w}^4$ is purely imaginary. For the derivative terms, we compute
\begin{equation*}
-d \Re \int_{\T} \abs{\tilde{u}}^2 (\partial_x \tilde{u}) \overline{\tilde{u}} dx = -d \int_{\T} \abs{\tilde{u}}^2 \frac{1}{2} \partial_x (\abs{\tilde{u}}^2) dx = -\frac{d}{4} \int_{\T} \partial_x (\abs{\tilde{u}}^4) dx = 0,
\end{equation*}
due to the periodic boundary conditions on $\T$. {The remaining term in \eqref{aggiunta} involving $e$ is handled similarly by observing that $\Re(\tilde{u}^2 \overline{\tilde{u}} \partial_x \overline{\tilde{u}}) = \abs{\tilde{u}}^2 \Re(\tilde{u} \partial_x \overline{\tilde{u}}) = \abs{\tilde{u}}^2 \Re(\overline{\tilde{u}} \partial_x \tilde{u})$, which reduces to the exact same total derivative $\frac{1}{4}\partial_x(\abs{\tilde{u}}^4)$ as the previous term.} This yields the exact conservation 
\[
\frac{d}{dt} \norm{u_N}_{L^2}^2 = 0,
\]
ensuring global existence.
\end{proof}

\subsection{Local existence for the original equation}
We can now prove a local-in-time existence result for the original Cauchy problem \eqref{eq:main_pde}, {which provides the first claim of Theorem \ref{GWPG}}.

\begin{lemma} \label{lem:local_existence_true_PDE}
{Fix $s > \frac{3}{2}$. For any initial datum $u_0 \in H^s$, and for any constant $C_1>1$, there exists $C_2>0$ such that the NLS-Dysthe equation \eqref{eq:main_pde} admits a unique solution $u \in C([-T^*, T^*], H^s)$ which satisfies the bound
\begin{equation} \label{eq:local_bound}
\sup_{t \in [-T^*, T^*]} \norm{u(t)}_{H^s} \le C_1 \norm{u_0}_{H^s},
\end{equation}
 with $T^*:=C_2 \|u_0\|_{H^s}^{-2}$, when $u_0 \neq 0$; otherwise $T^*=\infty$.} 
\end{lemma}
{The proof is obtained by the following lemmas.
\begin{lemma}\label{UNintermediate}
Fix $s > \frac{3}{2}$. For any initial datum $u_0 \in H^s$, and for any constant $C_1>1$, there exists $C_2>0$ such that the approximate Cauchy problem \eqref{eq:approx_problem} admits a unique solution $u_N \in C(\mathbb{R}, H^\infty)$ which satisfies the uniform bound
\begin{equation} \label{eq:claim_1}
 \sup_{t \in [-T^*, T^*]} \norm{u_N(t)}_{H^s} \le C_1 \norm{u_0}_{H^s} \qquad \forall N\geq 1,
\end{equation}
 with $T^*:=C_2 \|u_0\|_{H^s}^{-2}$, when $u_0 \neq 0$; otherwise $T^*=\infty$.
\end{lemma} }

\begin{proof}
Recalling the definitions in \eqref{eq:operators_def}, we compute the evolution of the $H^s$ norm to prove \eqref{eq:claim_1}. From Lemma  \ref{lem:global_approx}, for any $u_0 \in H^s$ there exists a unique solution $u_N \in C(\mathbb{R}, H^\infty)$ of \eqref{eq:approx_problem}. We compute
\begin{align*}
\frac{1}{2} \frac{d}{dt} \norm{u_N}_{H^s}^2& = \Re \inner{\La^s \partial_t u_N}{\La^s u_N}_{L^2} & \\&= \Re \inner{\La^s(A u_N)}{\La^s u_N}_{L^2} + \Re \inner{\La^s(J_N g(J_N u_N))}{\La^s u_N}_{L^2}.&
\end{align*}
The first term is identically zero because the operator $A$ is skew-adjoint and it commutes with $\La^s$. We therefore only need to carefully estimate the second term. Since the spectral projector $J_N$ is self-adjoint, it commutes with $\La^s$ and $J_N^2 = J_N$, we have $J_N u_N = u_N$, and we can rewrite the inner product as
\begin{equation}\label{stima}
\Re \inner{\La^s(J_N g(u_N))}{\La^s u_N}_{L^2} = \Re \inner{\La^s g(u_N)}{\La^s u_N}_{L^2}.
\end{equation}
By the explicit definition of $g(u_N)$ given in \eqref{eq:operators_def}, \eqref{stima} is composed of three distinct terms. Starting with the first one, applying Cauchy-Schwarz and the fractional Leibniz rule
\begin{equation}\label{leib}
\|\Lambda^s(fg)\|_{L^2} \lesssim \|f\|_{L^\infty}\|\Lambda^s g\|_{L^2}+\|g\|_{L^\infty}\|\Lambda^s f\|_{L^2}
\end{equation}
we have
\begin{equation*}
\begin{aligned}
\abs{\Re \int_{\T} \La^s(i b\abs{u_N}^2 u_N) \La^s \overline{u}_N dx}& \lesssim \norm{\La^s(\abs{u_N}^2 u_N)}_{L^2} \norm{\La^s u_N}_{L^2}& \\
&\lesssim \|u_N\|_{L^\infty}^2 \|u_N\|_{H^s}^2.&
\end{aligned}
\end{equation*}
Considering the second term of the nonlinearity $g$, we need to estimate
\begin{equation*}
I:=-d \,  \Re \inner{\La^s(\abs{u_N}^2 \partial_x u_N)}{\La^s u_N}_{L^2}.
\end{equation*}
If we try to apply a direct estimate exactly as before, we will lose a derivative. Therefore, we introduce the commutator operator 
\begin{equation*}
 [\La^s, \abs{u_N}^2] \partial_x u_N:= \La^s(\abs{u_N}^2 \partial_x u_N) - \abs{u_N}^2 \La^s(\partial_x u_N).
\end{equation*}
This substitution leads to
\begin{equation} 
\label{eq:deriv_term_split}
I=- d \, \Re \int_{\T} \abs{u_N}^2 \La^s(\partial_x u_N) \La^s \overline{u}_N dx -d \, \Re \int_{\T} [\La^s, \abs{u_N}^2](\partial_x u_N) \La^s \overline{u}_N dx.
\end{equation}
For the first integral on the right hand side, classical integration by parts yields
\begin{align}
\abs{ \Re \int_{\T} \abs{u_N}^2 (\partial_x \La^s u_N) \La^s \overline{u}_N dx} &= \abs{\int_{\T} \abs{u_N}^2 \frac{1}{2} \partial_x (\abs{\La^s u_N}^2) dx} \nonumber \\
&= \frac{1}{2}\abs{\int_{\T} \partial_x(\abs{u_N}^2) \abs{\La^s u_N}^2 dx} \nonumber \\
&\lesssim \norm{\partial_x(\abs{u_N}^2)}_{L^\infty} \norm{\La^s u_N}_{L^2}^2 \lesssim \|u_N\|_{W^{1,\infty}}^2\norm{u_N}_{H^s}^2. \label{eq:bound_comm_1}
\end{align}

For the second integral in \eqref{eq:deriv_term_split}, we rely on the Kato-Ponce commutator inequality
\begin{align*}
\norm{[\La^s, \abs{u_N}^2] \partial_x u_N}_{L^2} &\lesssim \norm{\partial_x(\abs{u_N}^2)}_{L^\infty} \norm{\La^{s-1}(\partial_x u_N)}_{L^2} + \norm{\La^s(\abs{u_N}^2)}_{L^2} \norm{\partial_x u_N}_{L^\infty} &\\
&\lesssim \|u_N\|_{W^{1,\infty}}^2 \|u_N\|_{H^s},&
\end{align*}
where we used \eqref{leib} again. For the Kato-Ponce estimates and the fractional Leibniz rule used throughout the paper, we refer respectively to \cite{chowdhury2025global,benyi2025fractional} . These references will be understood henceforth whenever such estimates are invoked.

A completely similar estimate can be obtained for the third term of $g$ involving the coefficient $e$. Combining these estimates and applying the embedding $H^s \hookrightarrow W^{1,\infty}$ valid for $s > \frac{3}{2}$,  we finally obtain the nonlinear differential inequality
\begin{equation} \label{eq:diff_ineq_Hs}
\left| \frac{d}{dt} \norm{u_N}_{H^s}^2 \right| \le K_0  \|u_N\|_{W^{1,\infty}}^2\norm{u_N}_{H^s}^2 \leq K_1 \|u_N\|_{H^s}^4,
\end{equation}
for some constants $K_0, K_1 > 0$, uniform in $N$. Integrating \eqref{eq:diff_ineq_Hs}, we obtain
\begin{equation} \label{eq:ODE_solution_Hs}
\norm{u_N(t)}_{H^s}^2 \le \frac{\norm{u_N(0)}_{H^s}^2}{1 - K_1 |t| \norm{u_N(0)}_{H^s}^2} \le \frac{\norm{u_0}_{H^s}^2}{1 - K_1 |t| \norm{u_0}_{H^s}^2}.
\end{equation}
{By taking $C_2:=(C_1^2-1)/(C_1^2K_1)$ we guarantee that \eqref{eq:claim_1} is satisfied for $T^*:=C_2\|u_0\|_{H^s}^{-2}$.}
\end{proof}

{\begin{lemma}\label{UNintermediate2}
Fix $s > \frac{3}{2}$ and initial datum $u_0 \in H^s$. Let $u_N \in C(\mathbb{R}, H^\infty)$ be a solution of \eqref{eq:approx_problem} such that, for some $T>0$, it satisfies the bound
\begin{equation}\label{ips} 
\sup_{t\in[-T,T]}\|u_N(t)\|_{H^s}\leq K \qquad \forall N\geq 1,
\end{equation}
for some $K>0$. Then, there exists a limit $u\in C([-T,T],H^s)$ such that $u_N\to u$ strongly in $C([-T,T],H^{s'})$ for any $0 \leq s' <s$ and $u_N(t)\rightharpoonup u(t)$ weakly in $H^s$ for all $t\in[-T,T]$. 
\end{lemma}}
\begin{proof}
Following an approach similar to \cite{hayashi2016well}, we prove that the truncated solutions form a Cauchy sequence. Fix any two indices $N, M \ge 1$ with $N > M$, and define the difference $w = u_N - u_M$. Then $w$ directly satisfies the equation
\begin{equation*}
\partial_t w = A w + J_N g(u_N) - J_M g(u_M),
\end{equation*}
with the operators $A$ and $g$ as previously defined in \eqref{eq:operators_def}. We compute the temporal evolution of the $L^2$ norm. Since $A$ is skew-adjoint, we have again
\begin{equation*}
\frac{1}{2} \frac{d}{dt} \norm{w}_{L^2}^2 = \Re \inner{\partial_t w}{w}_{L^2} = \Re \inner{J_N g(u_N) - J_M g(u_M)}{w}_{L^2}.
\end{equation*}
We algebraically decompose the nonlinear difference term as follows
\begin{equation*}
J_N g(u_N) - J_M g(u_M) = J_N(g(u_N) - g(u_M)) + (J_N - J_M) g(u_M).
\end{equation*}
Substituting this identity in the time derivative of the $L^2$ norm we obtain
\begin{equation} \label{eq:w_L2_evol}
\frac{1}{2} \frac{d}{dt} \norm{w}_{L^2}^2 = \Re \inner{g(u_N) - g(u_M)}{J_N w}_{L^2} + \Re \inner{(J_N - J_M)g(u_M)}{w}_{L^2}.
\end{equation}
For the second term in \eqref{eq:w_L2_evol}, we have 
\begin{align}
\abs{\Re \inner{(J_N - J_M)g(u_M)}{w}_{L^2}} &\le \norm{(J_N - J_M)g(u_M)}_{L^2} \norm{w}_{L^2} \nonumber \\
&\le \frac{1}{M^{s-1}} \norm{g(u_M)}_{H^{s-1}} \norm{w}_{L^2} \lesssim \frac{1}{M^{s-1}} \label{eq:error_term_M}
\end{align}
where in the third inequality we use the fourth property of Lemma \ref{lem:JN_properties}, while in the last inequality we use \eqref{ips} and \eqref{conservaL2} which give
\begin{equation}
\label{g:stim}
\begin{aligned}
\norm{g(u_M)}_{H^{s-1}} &\lesssim \norm{u_M}_{H^{s-1}} \norm{u_M}_{H^s}^2 \lesssim 1,&  \\
\norm{w}_{L^2} &\le \norm{u_N(0)}_{L^2} + \norm{u_M(0)}_{L^2} \lesssim \|u_0\|_{L^2},& 
\end{aligned}
\end{equation}
for any $t \in [-T,T]$ and the omitted constants are uniform in $N$ and $M$. Then in equation \eqref{eq:w_L2_evol}, this second term converges uniformly to zero as the index $M \to \infty$.

We now focus on the first crucial term in \eqref{eq:w_L2_evol}. We start by observing that
\begin{equation*}
\Re \inner{g(u_N) - g(u_M)}{J_N w} = \Re \inner{g(u_N) - g(u_M)}{w} + \Re \inner{g(u_N) - g(u_M)}{(I-J_N)u_M},
\end{equation*}
where the second error term is actually zero since $N>M$ and $J_Mu_M=u_M$.
Thus, we only need to consider the three contributions of $\Re \inner{g(u_N) - g(u_M)}{w}_{L^2}$. For example, analyzing the term involving $e$ in the definition of $g(u)$ in \eqref{eq:operators_def} we have 
\begin{align}
\nonumber
I_e &= -e \Re \int_{\T} (u_N^2 \partial_x \overline{u}_N - u_M^2 \partial_x \overline{u}_M) \overline{w} dx \\
\label{aux_1}
&= -e \Re \int_{\T} (u_N^2 \partial_x \overline{w}) \overline{w} dx - e \Re \int_{\T} (u_N + u_M) w (\partial_x \overline{u}_M) \overline{w} dx.
\end{align}
The first part of $I_e$ is easily estimated by integrating by parts
\begin{equation*}
\left|e \int_{\T} \Re \left( u_N^2 \frac{1}{2} \partial_x (\overline{w}^2) \right) dx \right|= \left|\frac{e}{2} \Re \int_{\T} \partial_x(u_N^2) \overline{w}^2 dx \right|\lesssim \norm{u_N}_{L^\infty} \norm{\partial_x u_N}_{L^\infty} \norm{w}_{L^2}^2.
\end{equation*}
While the second part in \eqref{aux_1} is directly bounded by 
\begin{equation*}
\left|e \Re \int_{\T} (u_N + u_M) w (\partial_x \overline{u}_M) \overline{w} dx \right|\lesssim(\norm{u_N}_{L^\infty} + \norm{u_M}_{L^\infty}) \norm{\partial_x u_M}_{L^\infty} \norm{w}_{L^2}^2\,.
\end{equation*}
By the Sobolev embedding $H^s \hookrightarrow W^{1,\infty}$ valid for $s > \frac{3}{2}$ and the uniform bound $K$, we obtain the differential inequality for $w$
\begin{equation} \label{eq:gronwall_w}
\left| \frac{d}{dt} \norm{w}_{L^2}^2 \right| \le K_2 \norm{w}_{L^2}^2 + \frac{K_3}{M^{s-1}},
\end{equation}
for some $K_3,K_2>0$ uniform in $M$ and $N$.
By applying Gronwall's lemma and using $w(0)=(J_N-J_M)u_0$, we obtain
\begin{equation*}
\norm{w(t)}_{L^2}^2 \le \norm{w(0)}_{L^2}^2 e^{K_2 |t|} + \frac{K_3}{K_2} \frac{1}{M^{s-1}} (e^{K_2 |t|} - 1) \lesssim \frac{e^{K_2 T}}{M^{s-1}} \qquad \forall t\in [-T, T].
\end{equation*}
 Therefore, $\sup_{[-T, T]} \norm{w(t)}_{L^2}^2 \to 0$ as $N, M \to \infty$. Thus, $\{u_N\}_{N \in \mathbb{N}}$ is a Cauchy sequence in $C([-T, T]; L^2)$ and it admits a strong limit $u$.

By \eqref{ips} we have that $\{u_N(t)\}_{N \in \mathbb{N}}$ is bounded in $H^s$ with $s>3/2$ uniformly for $t \in [-T, T]$. Then, there exists a function $v(t) \in H^s$ such that $u_N(t) \rightharpoonup v(t)$ weakly in $H^s$. But since $u_N \to u$ strongly in $C([-T, T], L^2)$, uniqueness of the limit implies $v(t) = u(t)$ for all $t  \in [-T, T]$. By interpolation, the convergence is strong for any $0 \leq s' <s$.

Finally, we want to prove that $u \in C([-T, T], H^s)$ (we adapt the argument in Remark (c) of \cite{kato1984nonlinear} to our setting). We recall that the $H^s$ norm is weakly lower semicontinuous with respect to time, {since it follows by standard arguments combining the $L^2$ strong convergence and the uniform bound in $H^s$. Then, for any $t \in [-T,T]$} we have
\begin{equation} \label{eq:liminf_norm}
\norm{u(t)}_{H^s}^2 \le \liminf_{\tau \to t} \norm{u(\tau)}_{H^s}^2
\end{equation}
 In particular, we have $\norm{u_0}_{H^s}^2 \le \liminf_{t \to 0} \norm{u(t)}_{H^s}^2$.
{Moreover, we observe that \eqref{eq:ODE_solution_Hs} is satisfied, and taking the lower limit as $N \to \infty$, the weak limit $u$ naturally satisfies
\begin{equation*}
\norm{u(t)}_{H^s}^2 \le \liminf_{N \to \infty} \norm{u_N(t)}_{H^s}^2 \le \frac{\norm{u_0}_{H^s}^2}{1 -  K_1 \|u_0\|_{H^s}^2 |t|}.
\end{equation*}
Taking the upper limit as $t \to 0$, we get:
\begin{equation} \label{eq:limsup_norm}
\limsup_{t \to 0} \norm{u(t)}_{H^s}^2 \le \norm{u_0}_{H^s}^2
\end{equation}}
{Combining \eqref{eq:liminf_norm} and \eqref{eq:limsup_norm} yields
\[
\|u(t)\|_{H^s}\to\|u_0\|_{H^s} \qquad\text{as }t\to0.
\]
Together with the weak convergence in $H^s$, this implies $u(t)\to u_0$ strongly in $H^s$. The same argument can be restarted at any $t_0\in[-T,T]$. Indeed, let $v_N$ be the solution of the approximate problem with $v_N(t_0)=J_Nu(t_0)$. Since
\[
\|J_Nu(t_0)-u_N(t_0)\|_{L^2} \le \|(J_N-I)u(t_0)\|_{L^2}    +\|u(t_0)-u_N(t_0)\|_{L^2}\to0,
\]
the same estimates leading to \eqref{eq:gronwall_w}, applied to two solutions of the same $N$-truncated equation (so that the projection-error term is absent), yield
\[
\sup_{t\in [t_0-\delta, t_0+\delta]}\|v_N(t)-u_N(t)\|_{L^2}\to 0,
\]
for some universal constant $\delta>0$. Hence $v_N\to u$ in $C([t_0-\delta, t_0+\delta]; L^2)$. Since by Lemma \ref{lem:JN_properties}, it holds $J_Nu(t_0)\to u(t_0)$ strongly in $H^s$, the preceding argument  with $t_0$ in place of $0$, gives strong $H^s$ continuity at $t_0$. We remark that the restart interval can be chosen uniform in $N$ since $\|J_Nu(t_0)\|_{H^s} \leq \|u(t_0)\|_{H^s}$. }
\end{proof}

\begin{proof}[Proof of Lemma \ref{lem:local_existence_true_PDE}]
First we observe that for $T^*:=C_2\|u_0\|_{H^s}^{-2}$ introduced in Lemma \ref{UNintermediate}, \eqref{eq:claim_1} holds. Then we can apply Lemma \ref{UNintermediate2} to state that there exists  $u \in C([-T^*, T^*], H^s)$ such that $u_N \to u$ strongly in $ C([-T^*, T^*], H^{s'})$ for any $0\leq s'<s$ and weakly in $H^s$. Moreover,  by the weak lower semicontinuity of the
$H^s$ norm, for every $t\in[-T^*,T^*]$ we have
\[
\|u(t)\|_{H^s} \leq \liminf_{N \to \infty} \|u_N(t)\|_{H^s} \leq \sup_{N \ge 1} \norm{u_N(t)}_{H^s} \leq C_1 \|u_0\|_{H^s},
\]
where we applied \eqref{eq:claim_1}. Since the previous inequality holds true for any   $t \in [-T^*,T^*]$, we obtain \eqref{eq:local_bound}.

{To conclude the proof of Lemma \ref{lem:local_existence_true_PDE}, it remains to verify that the limit function $u \in C([-T^*, T^*], H^s)$ satisfies the Duhamel formula and it is unique. From the approximate problem, we know that
\begin{equation}\label{duafin_approx}
u_N(t) = U(t)J_N u_0 + \int_0^t U(t-s) J_N g(u_N(s)) ds,
\end{equation}
where $U(t) = e^{tA}$ is the unitary group generated by the skew-adjoint operator $A$ (see \eqref{eq:operators_def}). To pass to the limit as $N \to \infty$, we observe that for the nonlinear term we have
\begin{equation}\label{eq:limit_g}
\|J_N g(u_N) - g(u)\|_{L^2} \leq \|J_N (g(u_N) - g(u))\|_{L^2} + \|(I - J_N)g(u)\|_{L^2}.
\end{equation}
{
We claim that the nonlinear map $g$ is locally Lipschitz from $H^{s'}$ to $L^2$ for any $s' > 3/2$. Indeed, by performing the same estimates used to obtain \eqref{eq:K_lipschitz}, with the only difference of bounding the $W^{1,\infty}$ norm directly by the respective $H^{s'}$ bounds via Sobolev embeddings (and not using the cut-offs), for any $u,v \in H^{s'}$ we have
\[
\|g(u)-g(v)\|_{L^2} \lesssim \|u-v\|_{H^{s'}}.
\]
}
Since $u_N \to u$ strongly in $C([-T^*, T^*], H^{s'}(\T))$ for $3/2< s' < s$, the first term on the right-hand side of \eqref{eq:limit_g} vanishes uniformly in time. The second term strongly converges to zero by the standard properties of the spectral projector (Lemma \ref{lem:JN_properties}). Because $U(t)$ is an isometry on $L^2$, we can take the uniform limit $N \to \infty$ in \eqref{duafin_approx} to conclude that $u(t)$ satisfies the integral equation (in the $L^2$ sense) 
\[
u(t) = U(t) u_0 + \int_0^t U(t-s)g(u(s)) ds, \qquad \forall t \in [-T^*, T^*].
\]
Finally, we study the uniqueness. { Let $u,v \in C([-T, T], H^s)$ be two solutions for some $T>0$. We are going to prove the more general claim 
\begin{equation}\label{lip}
\sup_{t \in [-T,T]} \|u-v\|_{L^2} \lesssim \|u(0)-v(0)\|_{L^2}.
\end{equation}
Since they are continuous in time, their $H^s$ norms are bounded on $[-T, T]$.} 
{Since $s>3/2$, if $w=u-v$, we have
$Aw\in H^{s-3}$ and
$w\in H^s\hookrightarrow H^{3-s}$. By the skew-adjointness of $A$
\[
\Re\langle Aw,w\rangle=0,
\]
where the product is well defined as the duality pairing between $H^{s-3}$ and $H^{3-s}$. This duality consideration is omitted from now on.
}

We then evaluate the time evolution of the $L^2$ distance
\begin{equation*}
\frac{1}{2} \frac{d}{d t} \|u-v\|_{L^2}^2 = \Re \inner{g(u)-g(v)}{u-v}_{L^2} \,.
\end{equation*}

We claim that
\begin{equation*}
| \Re \inner{g(u)-g(v)}{u-v}_{L^2} | \lesssim \| u-v\|^2_{L^2}\,.
\end{equation*}
We show how to bound only the terms involving a derivative in the expression for $g$ (see \eqref{eq:operators_def}). We have, for instance, integrating by parts,
\begin{align*}
 &\Re \int_{\T} (|u|^2 u_x-|v|^2 v_x) \overline{u-v}\,dx= \Re\int_{\T} |u|^2 (u-v)_x \overline{u-v}dx+ \Re\int_{\T} (|u|^2-|v|^2) v_x \overline{u-v}\,dx\\
 &=-\frac{1}{2}\int_{\T} \partial_x(|u|^2) |u-v|^2dx+\Re\int_{\T} (|u|^2-|v|^2) v_x \overline{u-v}\,dx\\
 &=-\frac{1}{2}\int_{\T} \partial_x(|u|^2) |u-v|^2dx+\Re\int_{\T} u v_x (\overline{u-v})^2\,dx+\Re\int_{\T} \overline{v} v_x |u-v|^2dx.
\end{align*}
It is sufficient now to bound the terms with a derivative in $L^{\infty}$ and use Cauchy-Schwarz to estimate the remaining integrals.
This yields the differential inequality
\begin{equation*}
\left| \frac{d}{d t} \|u-v\|_{L^2}^2 \right| \lesssim \|u-v\|_{L^2}^2,
\end{equation*}
which, by Gronwall's lemma gives \eqref{lip}. In particular, applying this estimate on
$[-T^*,T^*]$ to two solutions with the same
initial datum $u_0$, we conclude that they coincide
throughout this interval, proving uniqueness.}
\end{proof}
\subsection{Continuous Dependence on Initial Data}
The principal obstacle in establishing continuous dependence on the initial data stems from the unboundedness of the nonlinearity in the equation \eqref{eq:main_pde}. Evaluating the difference between two solutions $w = u - v$ and estimating its $H^s$ norm via the fractional Leibniz rule yields a bound involving the $H^{s+1}$ norm, resulting in an apparent loss of one derivative. 
To overcome this derivative loss, we employ the Bona-Smith approximation method, originally introduced for the Korteweg-de Vries equation in \cite{bona1975initial}. We consider initial data truncated at frequency $N$. The corresponding regularized solutions are smooth, and their $H^{s+1}$ norms are bounded by $O(N)$. By interpolating the $L^\infty$ norm of the difference, we can carefully balance this $\mathcal{O}(N)$ high-frequency divergence against the $\mathcal{O}(N^{-s})$ low-frequency convergence of the initial data. 

We state the main theorem regarding the continuous dependence on the initial data.
\begin{theorem}\label{thm:continuous_dependence}
Let $s > \frac{3}{2}$. Let $u_0 \in H^s$ and $\{u_{0,k}\}_{k \in \mathbb{N}}\subset H^s$ be a sequence such that $u_{0,k} \to u_0$ strongly in $H^s$. Let $u$ and $u_k$ be the corresponding maximal solutions to the Cauchy problem \eqref{eq:main_pde}, defined on $[0, T_\infty)$ and $[0, T_k)$ respectively. {For any $T \in (0,T_\infty)$, and for every $\varepsilon > 0$, there exists $K \in \mathbb{N}$ such that for all $k > K$, we have $T_k > T$ and
\begin{equation*}
\sup_{t \in [0, T]} \norm{u(t) - u_k(t)}_{H^s} \le \varepsilon.
\end{equation*}}
\end{theorem}
\begin{remark}
Since equation \eqref{eq:main_pde} is invariant under
the transformation
\[
u(t,x)\mapsto\overline{u(-t,-x)},
\]
the same result holds for negative times.
\end{remark}
To prove this theorem, we split the analysis into two intermediate lemmas. We first establish the exact growth rate of the high-frequency norm for the regularized solutions.
\begin{lemma} \label{lem:high_freq_bound}
{Fix $s > \frac{3}{2}$. Let $u_0 \in H^s$ and let $\mu^N \in C([0, T_{\max}); H^\infty)$ be the unique maximal solution to \eqref{eq:main_pde} with regularized initial data $\mu^N(0) = J_N u_0$, for some $T_{\max}>0$. For any $C_1>1$, there exist $C_2>0$ and $C_3>1$ uniform in $N$ such that 
\begin{equation} \label{eq:Hs1_bound}
\sup_{t \in [0, \tilde{T}]} \norm{\mu^N(t)}_{H^{s+1}} \le C_{3} N \norm{u_0}_{H^s}, \qquad \sup_{t \in [0, \tilde{T}]} \norm{\mu^N(t)}_{H^{s}} \le C_1 \norm{u_0}_{H^s},
\end{equation}
where $\tilde{T}:=C_2 \|u_0\|_{H^s}^{-2}$ when $u_0\neq 0$; otherwise $\tilde{T}=+\infty$.} 
\end{lemma}
\begin{proof}
Applying the Bessel potential operator $\La^{s+1} = (1-\partial_x^2)^{(s+1)/2}$ to the equation and proceeding exactly as in the proof of Lemma \ref{lem:local_existence_true_PDE}, we isolate the highest-order derivatives.  Thanks to Lemma \ref{lem:local_existence_true_PDE}, for any $C_1>1$ and $u_0 \in H^s$ we have the uniform bound 
\begin{equation}\label{baundmun}
\sup_{[0, T^*]} \norm{\mu^N}_{H^s} \le C_1 \norm{J_N u_0}_{H^s} \leq C_1 \|u_0\|_{H^s}, \qquad s>3/2, \qquad \forall N\geq 1,
\end{equation}
{for some $C_2>0$ independent of the initial datum and $N$, and $T^*:= C_2 \|J_N u_0\|_{H^s}^{-2}.$ Then, using the cut-off properties, we define the time $\tilde{T}:=C_2 \|u_0\|_{H^s}^{-2} \leq T^*$, which is uniform in $N$, and \eqref{baundmun} holds true also on $[0, \tilde{T}]$. }

Then, for any $t \in [0, \tilde{T}]$ the norm $\norm{\mu^N}_{W^{1,\infty}}$ is bounded uniformly in $N$ and applying the first inequality in \eqref{eq:diff_ineq_Hs} for $H^{s+1}$, we have
\begin{equation*}
\frac{d}{dt} \norm{\mu^N}_{H^{s+1}}^2 \lesssim \norm{\mu^N}_{H^{s+1}}^2.
\end{equation*}
Integrating this inequality via Gronwall's lemma, using the third property in Lemma \ref{lem:JN_properties}, for any $t \in [0,\tilde{T}]$ there exists $C_3>0$, independent of $N$, such that
\[
\norm{\mu^N(t)}_{H^{s+1}}\le C_3 \| J_N u_0\|_{H^{s+1}} \leq \sqrt{2} C_3 N\|u_0\|_{H^s},
\]
which concludes the proof.
\end{proof}
We now state a generalized difference estimate that encapsulates the core of the Bona-Smith argument.
\begin{lemma} \label{lem:generic_difference}
{Let $s > \frac{3}{2}$ and $T_0>0$. Let $u \in C([0, T_0]; H^{s+1})$ and $v  \in C([0, T_0]; H^{s})$ be two solutions to \eqref{eq:main_pde}. Then, for any $\theta \in (\frac{1}{s}, 1 - \frac{1}{2s})$, the difference $w = u - v$ satisfies the estimate 
\begin{equation}\label{utile:cont}
\sup_{t \in [0, T_0]} \norm{w(t)}_{H^s} \le C \left( \norm{w(0)}_{H^s} + \sup_{t \in [0, T_0]} \norm{u(t)}_{H^{s+1}}^{\frac{1}{\theta}}  \norm{w(0)}_{L^2} \right),
\end{equation}
where $C > 0$ depends only on $T_0$, $\theta$ and 
\[
\sup_{t \in [0, T_0]} (\norm{u(t)}_{H^s} + \norm{v(t)}_{H^s}).
\]}
\end{lemma}
\begin{proof}
We compute the time evolution of the $H^s$ norm of the difference $w$, understood through a standard spatial regularization argument, as follows
\begin{equation}\label{stima:continua}
\frac{1}{2} \frac{d}{dt} \norm{w}_{H^s}^2 = \Re \inner{\La^s(g(u) - g(v))}{\La^s w}_{L^2}.
\end{equation}
We decompose the differences to place the spatial derivative on the smoother solution $u$. For the term in the nonlinearity $g$ involving the coefficient $d$ (see \eqref{eq:operators_def}), we write
\begin{equation}\label{d:parte}
\abs{u}^2 \partial_x u - \abs{v}^2 \partial_x v = \abs{v}^2 \partial_x w + (u \overline{w} + w \overline{v}) \partial_x u.
\end{equation}
The contribution of the first term in the r.h.s. of \eqref{d:parte} in \eqref{stima:continua} is expanded via the commutator as follows
\begin{equation*}
-d \Re \inner{\La^s(\abs{v}^2 \partial_x w)}{\La^s w}_{L^2} = -d \Re \int_{\T} \abs{v}^2 (\partial_x \La^s w) \overline{\La^s w} dx - d \Re \inner{[\La^s, \abs{v}^2] \partial_x w}{\La^s w}_{L^2}.
\end{equation*}
Integrating the first integral by parts yields 
\[
-d \Re \int_{\T} \abs{v}^2 (\partial_x \La^s w) \overline{\La^s w} dx=\frac{d}{2} \int \partial_x(\abs{v}^2) \abs{\La^s w}^2 dx \lesssim \norm{v}^2_{W^{1,\infty}} \norm{w}_{H^s}^2 \lesssim \norm{w}_{H^s}^2.
\]
For the commutator, since $s > \frac{3}{2}$, the Kato-Ponce inequality provides
\begin{align*}
\norm{[\La^s, \abs{v}^2] \partial_x w}_{L^2} &\lesssim \norm{\partial_x(\abs{v}^2)}_{L^\infty} \norm{\La^{s-1} \partial_x w}_{L^2} + \norm{\La^s(\abs{v}^2)}_{L^2} \norm{\partial_x w}_{L^\infty} &\\
& \lesssim \norm{v}_{W^{1,\infty}}^2 \norm{w}_{H^s}+\| v\|_{H^s}^2 \| w \|_{W^{1, \infty}} \lesssim \norm{w}_{H^s}.&
\end{align*}
We now consider the contribution of the second term in the r.h.s. of \eqref{d:parte} in \eqref{stima:continua}. We define $F := u \overline{w} + w \overline{v}$ and by the fractional Leibniz rule, we have 
\begin{equation*}
\norm{\La^s(F \partial_x u)}_{L^2} \lesssim \norm{F}_{L^\infty} \norm{\La^s \partial_x u}_{L^2} + \norm{\La^s F}_{L^2} \norm{\partial_x u}_{L^\infty}.
\end{equation*}
Using the algebra property of $H^s$ and Sobolev embeddings, $\norm{\La^s F}_{L^2} \lesssim \norm{w}_{H^s}$ and $\norm{\partial_x u}_{L^\infty} \lesssim \|u\|_{H^s}$. Crucially, the derivative loss is localized on $u$, giving $\norm{\La^s \partial_x u}_{L^2} \le \norm{u}_{H^{s+1}}$, which is bounded by hypothesis. {To estimate $\|F\|_{L^\infty}$ we apply the one-dimensional Gagliardo-Nirenberg inequality. First, we recall that since $s > 3/2$, the interval $I_s:=(\frac{1}{s}, 1 - \frac{1}{2s})$ is non-empty, allowing us to fix  $\theta \in I_s$. By the Gagliardo-Nirenberg inequality we have $\norm{w}_{L^\infty} \lesssim \norm{w}_{L^2}^{\theta_0} \norm{w}_{H^s}^{1-\theta_0}$ for any $\theta_0\in (0, 1)$ satisfying $(1-\theta_0)s > 1/2$. Since any $\theta \in I_s$ satisfies the previous bound, we obtain $\norm{F}_{L^\infty} \lesssim  \norm{w}_{L^2}^\theta \norm{w}_{H^s}^{1-\theta}$. }
The terms involving the coefficients $b$ and $e$ (see \eqref{eq:operators_def}) are treated symmetrically. Aggregating all bounds yields the differential inequality
\begin{equation*}
\frac{1}{2} \frac{d}{dt} \norm{w}_{H^s}^2 \lesssim \norm{w}_{H^s}^2 + \norm{u}_{H^{s+1}} \norm{w}_{L^2}^\theta \norm{w}_{H^s}^{2-\theta}.
\end{equation*}

If $w(t_0)=0$ for some $t_0\in[0,T_0]$, uniqueness
implies $w\equiv 0$, and the conclusion is immediate.
We may therefore assume that $\|w(t)\|_{H^s}>0$
throughout $[0,T_0]$.
Dividing the preceding differential inequality by
$\|w\|_{H^s}^{2-\theta}$, we obtain
\[
\frac{d}{dt}\big(\|w\|_{H^s}^{\theta}\big)
\lesssim
\|w\|_{H^s}^{\theta}
+\|u\|_{H^{s+1}}\|w\|_{L^2}^{\theta}.
\]
Using the $L^2$ stability estimate \eqref{lip} and
applying Gronwall's inequality, we deduce
\[
\sup_{t\in[0,T_0]}\|w(t)\|_{H^s}^{\theta}
\lesssim
\|w(0)\|_{H^s}^{\theta}
+
\sup_{t\in[0,T_0]}\|u(t)\|_{H^{s+1}}
\|w(0)\|_{L^2}^{\theta}.
\]
Raising both sides to the power $1/\theta$
yields \eqref{utile:cont}, completing the proof.

\end{proof}
\begin{proof}[Proof of Theorem \ref{thm:continuous_dependence}]
Let $u_0 \in H^s$ and $\{u_{0,k}\}_{k \in \mathbb{N}}\subset H^s$ be such that $u_{0,k} \to u_0$ strongly in $H^s$. Let $u$ and $u_k$ be the corresponding maximal solutions to the Cauchy problem \eqref{eq:main_pde}, defined on $[0, T_\infty)$ and $[0, T_k)$ respectively. 

\medskip

{We define $M_0$ as 
\begin{equation*}
M_0 \coloneqq  \norm{u_0}_{H^s} + 1 < \infty.
\end{equation*}
By the local existence theory provided by Lemma \ref{lem:local_existence_true_PDE}, the lifespan of a solution to the Cauchy problem \eqref{eq:main_pde} with an initial datum of $H^s$-size less than or equal to $M_0>0$ is at least $T_{M_0} := C_0 M_0^{-2} > 0$, for some constant $C_0>0$ independent of the initial datum. 
Since $u_{0,k} \to u_0$ strongly in $H^s$, there exists $K_0>0$ such that $\norm{u_{0,k}}_{H^s} \le M_0$ for any $k \geq K_0$, ensuring that their individual existence times satisfy $T_k > T_{M_0}$, and $T_\infty> T_{M_0}$. This proves that $\inf_{k\geq K_0} T_k>0$.
We now fix $T \in [0,T_\infty)$ and we define the global safety threshold $M_1>0$ 
\begin{equation}\label{M1}
M_1:= \sup_{t \in [0,T]} \|u(t)\|_{H^s}+1< \infty.
\end{equation}
Reasoning as above, since $M_1 \geq M_0$, for any $k\geq K_0$ as above, $T_\infty > T_{M_1}:=C_0 M_1^{-2}>0$. If $T> T_{M_1}$,} we partition the macroscopic interval $[0, T]$ into exact iterations 
\[
\tau \coloneqq \left\lfloor \frac{T}{T_{M_1}} \right\rfloor + 1
\]
 representing subintervals of length at most $T_{M_1}$. Let $t_i = i T_{M_1}$ for $i = 0, \dots, \tau-1$ and $t_\tau = T$. On the other hand, if $ T \leq T_{M_1}$, we perform a single iteration. 

\medskip

We introduce the regularized data $J_N u_0$ and $J_N u_{0,k}$, and denote $\mu^N$ and $\mu_k^N$ as their respective solutions in $C([0,T_{M_1}],H^{s+1})$. By Lemma \ref{lem:high_freq_bound} yields the high-frequency bounds 
\[
\sup_{t\in[0,T_{M_1}]}
\|\mu^N(t)\|_{H^{s+1}}
\leq C N\|u_0\|_{H^s},
\qquad
\sup_{k\geq K_0}\sup_{t\in[0,T_{M_1}]}
\|\mu_k^N(t)\|_{H^{s+1}}
\leq C N\sup_{k\geq K_0}\|u_{0,k}\|_{H^s}.
\]
By the triangle inequality, we decompose the difference into three distinct parts on $[0, T_{M_1} ]$
\begin{equation} \label{eq:triangle_split_0}
\norm{u_k - u}_{H^s} \le \norm{u_k - \mu_k^N}_{H^s} + \norm{\mu_k^N - \mu^N}_{H^s} + \norm{\mu^N - u}_{H^s}.
\end{equation}
Applying Lemma \ref{lem:generic_difference} to the regularization errors (assigning $\mu_k^N$ and $\mu^N$ as smooth reference solutions), we obtain
\begin{align}
\label{1:cont}
\sup_{t \in [0, T_{M_1}]} \norm{u_k - \mu_k^N}_{H^s} &\lesssim \left( \norm{(I - J_N) u_{0,k}}_{H^s} + N^{\frac{1}{\theta}-s} \norm{u_{0,k}}_{H^s}^{1+\frac{1}{\theta}} \right), \\
\label{2:cont}
\sup_{t \in [0, T_{M_1}]} \norm{\mu^N - u}_{H^s} &\lesssim \left( \norm{(I - J_N) u_0}_{H^s} + N^{\frac{1}{\theta}-s} \norm{u_0}_{H^s}^{1+\frac{1}{\theta}} \right),
\end{align}
where we used the fourth property of Lemma \ref{lem:JN_properties} to bound the third norm in the r.h.s. of \eqref{utile:cont}. 

{We recall that $\theta > \frac{1}{s}$, so the critical exponent $\frac{1}{\theta} - s$ is strictly negative. Furthermore, the strong convergence $u_{0,k} \to u_0$ in $H^s$ guarantees that $\{u_{0,k}\}_{k \in \mathbb{N}}$ is uniformly bounded in $H^s$. Therefore, we can make $N^{\frac{1}{\theta}-s}\Vert{}u_0\Vert{}_{H^s} + N^{\frac{1}{\theta}-s} \sup_{k \in \mathbb{N}} \Vert{}u_{0,k}\Vert{}_{H^s}$ arbitrarily small by choosing a cutoff $N$ sufficiently large, depending only on $u_0$ and the uniform $H^s$-bound of $\{u_{0,k}\}_k$. Similarly, since the first term in the r.h.s. of \eqref{2:cont} is a tail that goes to zero for $N \to \infty$, it can also be made arbitrarily small. For the first term on the r.h.s. of \eqref{1:cont}, we combine the previous limit with the convergence $u_{0,k} \to u_0$: by a triangle inequality argument, this is controlled by choosing two independent $N, K_1 \in \mathbb{N}$ large enough, making also this term arbitrarily small for $N$ and any $k \geq K_1$. Consequently, for any $\eps>0$, there exist $N_\eps$ and $K_1$ such that \eqref{1:cont} and \eqref{2:cont} are bounded by $\eps/3$ for all $k \ge K_1$.}

With $N_\eps$ fixed as above, we evaluate the middle pair using Lemma \ref{lem:generic_difference} with $\mu^N$ as the smooth reference, to obtain
\begin{equation*}
\sup_{t \in [0, T_{M_1}]} \norm{\mu_k^N - \mu^N}_{H^s} \lesssim \left( \norm{J_N (u_{0,k} - u_0)}_{H^s} + N^{\frac{1}{\theta}} \norm{J_N (u_{0,k} - u_0)}_{L^2} \right).
\end{equation*}
{Since $u_{0,k} \to u_0$ in $H^s$, there exists an index $K_2 \in \mathbb{N}$ such that for all $k \ge K_2$, this middle term is less than $\eps/3$. Summing the three contributions in \eqref{eq:triangle_split_0} for $N=N_\eps$ and $k> K_{0,\eps}:= \max \{K_1,K_2,K_0\}$, we conclude}
\begin{equation} \label{eq:eps_K1_bound}
\sup_{t \in [0, T_{M_1}]} \norm{u_k(t) - u(t)}_{H^s} <\eps,
\end{equation}
and then $u_k \to u$ in $C([0,T_{M_1}],H^s)$. 

\medskip

For any $\varepsilon \in (0, 1)$, by the definition of $M_1$ and the bound \eqref{eq:eps_K1_bound} evaluated exactly at $t = T_{M_1}$, for all $k > K_{0,\eps}$ 
\begin{equation}\label{T1}
\norm{u_k(T_{M_1})}_{H^s} \le \norm{u(T_{M_1})}_{H^s} + \norm{u_k(T_{M_1}) - u(T_{M_1})}_{H^s} \le (M_1 - 1) +\eps \le M_1.
\end{equation}
Since $\norm{u_k(T_{M_1})}_{H^s} \le M_1$, the local existence theory (Lemma \ref{lem:local_existence_true_PDE}) guarantees that $u_k$ can be extended for at least another time step $T_{M_1}$, so that it exists at least up to $2T_{M_1}$. 
We can systematically iterate this argument over the $\tau$ intervals $[t_i, t_{i+1}]$. Assuming inductively that $u_k \to u$ in $C([t_{i-1}, t_i],H^s)$, we naturally find again $\norm{u_k(t_i)}_{H^s} \le M_1$, as we did in \eqref{T1}. 
Repeating the Bona-Smith approximation argument leading to \eqref{eq:eps_K1_bound}, now with initial data $u(t_i)$ and $u_k(t_i)$, we obtain,
 for any $\eps>0$, an index $K_{i+1,\eps} \ge K_{i, \eps}$ such that the error on $[t_i, t_{i+1}]$ is bounded by $\eps$.
{After exactly $\tau$ iterations, defining $K = K_{\tau, \eps}$, the solutions $u_k$ with $k> K_{\tau, \eps}$ are well defined on $[0, T]$  (then $T_k \geq T$)} and taking the maximum over the finitely many intervals the error satisfies
\begin{equation*}
\sup_{t \in [0, T]} \norm{u_k(t) - u(t)}_{H^s} < \eps\,.
\end{equation*}
Since $T \in (0, T_\infty)$ and $\varepsilon > 0$ were arbitrary, this completes the proof of the continuous dependence on the initial data.
\end{proof}

\section{Conditional global existence in $H^2$}\label{MDE}
\subsection{Modified energy}
We recall the definition of the \emph{good functionals} introduced in \cite{hayashi2025global}.

\begin{definition} \label{def:G_set}
We denote by $\mathcal{G}$ the set of all functionals $G \in C(H^{2}, \R)$ such that for all $M>0$ there exists a constant $C(M)>0$ satisfying
\begin{equation}\label{esti:G}
u \in H^{2}, \quad \norm{u}_{H^{1}} \le M \implies \abs{G(u)} \le C(M) (1 + \norm{u}_{H^{2}}^{2}).
\end{equation}
\end{definition}

This functional class is introduced to rigorously identify admissible lower-order perturbations of the standard $H^2$ energy. If we can construct a modified energy equivalent to the $H^2$ norm up to terms in $\mathcal{G}$, we can safely apply Gronwall's lemma to prevent finite-time blow-up, assuming an \textit{a priori} bound on the $H^1$ norm. The core idea is to construct suitable energy corrections whose time derivatives cancel the higher-order contributions, leaving only terms in $\mathcal G$. We restrict our construction to $s=2$ since the explicit algebraic cancellations of the modified energy are highly sensitive to the integer number of derivatives.

\smallskip

Let $N>1$ and $u_{N} \in C(\R, H^{\infty})$ be the approximate solution obtained in Lemma \ref{lem:global_approx}. By the properties of the spectral projector, we know that $u_N = J_N u_{N}$. 
\begin{theorem} \label{thm:energy}
Let $N>1$ and $u_{N} \in C(\R, H^{\infty})$ be the solution of \eqref{eq:approx_problem}. Then there exists a functional $G_{rem,N} \in \mathcal{G}$, with the estimate \eqref{esti:G} uniform in $N$, such that the modified energy defined by
\begin{equation} \label{eq:modified_energy_def2}
\mathcal{E}_N(u_N) = \norm{\partial_{x}^{2} u_N}_{L^{2}}^{2} - \frac{4d+2e}{3c} \int \abs{u_N}^2 \abs{\partial_x u_N}^2 dx - \frac{d+3e}{3c} \int \Re\left(u_N^2 (\partial_x \overline{u}_N)^2\right) dx
\end{equation}
satisfies the differential identity
\begin{equation} \label{eq:modified_energy_evol}
\frac{d}{dt} \mathcal{E}_N(u_N) = G_{rem,N}(u_N).
\end{equation}
\end{theorem}

We introduce the notation
\begin{equation}\label{C1C2}
\mathcal{C}_1(u_N):=\int \abs{u_N}^2 \abs{\partial_x u_N}^2 dx, \qquad \mathcal{C}_2(u_N):= \int \Re\left(u_N^2 (\partial_x \overline{u}_N)^2\right) dx\,
\end{equation}
so that
\begin{equation}\label{eq:modified_energy_def}
\mathcal{E}_N(u_N)=\norm{\partial_{x}^{2} u_N}_{L^{2}}^{2} - \frac{4d+2e}{3c} \mathcal{C}_1(u_N)- \frac{d+3e}{3c} \mathcal{C}_2(u_N)\,.
\end{equation}

In order to systematically prove Theorem \ref{thm:energy}, we split the analysis into three intermediate lemmas. We will compute the time derivative of each component of the modified energy separately, isolating the problematic higher-order terms. 

\begin{lemma} \label{lem:energy_part1}
Let $N>1$ and $u_{N} \in C(\R, H^{\infty})$ be the solution of \eqref{eq:approx_problem}. Then there exists a remainder $G_{1} \in \mathcal{G}$, with the estimate \eqref{esti:G} uniform in $N$, such that
\begin{equation} \label{eq:evol_part1}
\frac{1}{2} \frac{d}{dt} \norm{\partial_{x}^{2} u_{N}}_{L^{2}}^{2} =- (4d+2e)T_{1} - (d+3e)T_{2} + G_{1}(u_N),
\end{equation}
where the leading-order terms related to the derivative loss are defined as
\begin{align}
T_{1} &:= \int \abs{\partial_x^2 u_N}^2 \Re(u_N \partial_x \overline{u}_N) dx, \label{eq:T1_def} \\
T_{2} &:= \int \Re\left( u_N \partial_x u_N (\partial_x^2 \overline{u}_N)^2 \right) dx. \label{eq:T2_def}
\end{align}
\end{lemma}

\begin{proof}
We compute the temporal evolution of the $L^2$ norm of the second spatial derivative of $u_N$
\begin{align}
\frac{1}{2} \frac{d}{dt} \norm{\partial_{x}^{2} u_{N}}_{L^{2}}^{2} &= \Re \inner{\partial_t u_N}{\partial_x^4 u_N}_{L^2} \nonumber \\
&= \Re \inner{A u_N}{\partial_x^4 u_N}_{L^2} + \Re \inner{J_N g(J_N u_N)}{\partial_x^4 u_N}_{L^2}. \label{eq:H2_evol_split}
\end{align}
For the linear component, repeated integration by parts yields
\begin{equation*}
\Re \inner{A u_N}{\partial_x^4 u_N}_{L^2} = \Re \inner{A \partial_x^2 u_N}{\partial_x^2 u_N}_{L^2} = 0,
\end{equation*}
due to the fact that $A = i a\partial_x^2 - c\partial_x^3$ commutes with $\partial_x$ and it is skew-adjoint.
For the nonlinear part in \eqref{eq:H2_evol_split}, using the self-adjointness of $J_N$ and recalling $u_N = J_N u_N$, we define
\begin{equation} \label{eq:I_term}
I := \Re \inner{J_N g(J_N u_N)}{\partial_x^4 u_N}_{L^2} = \Re \inner{g(u_N)}{\partial_x^4 u_N}_{L^2}.
\end{equation}
We now analyze the contribution of each term inside the nonlinearity $g(u_N)$ as defined in \eqref{eq:operators_def}. Starting with the purely cubic term 
\begin{equation} \label{eq:Ib_term}
I_b := \Re \inner{ib \abs{u_N}^2 u_N}{\partial_x^4 u_N}_{L^2} = \Re \inner{ib \partial_x^2 (\abs{u_N}^2 u_N)}{\partial_x^2 u_N}_{L^2}.
\end{equation}
By direct computation, we expand the second derivative
\begin{equation*}
\partial_x^2 (\abs{u_N}^2 u_N) = 2 \abs{u_N}^2 \partial_x^2 u_N + u_N^2 \partial_x^2 \overline{u}_N + 4 u_N \abs{\partial_x u_N}^2 + 2 \overline{u}_N (\partial_x u_N)^2.
\end{equation*}
Consequently, $I_b$ breaks down into four pieces. The first one identically vanishes since the integrand is purely imaginary
\begin{equation*}
\Re \inner{2ib \abs{u_N}^2 \partial_x^2 u_N}{\partial_x^2 u_N}_{L^2} = 2b \Re \int i \abs{u_N}^2 \abs{\partial_x^2 u_N}^2 dx = 0.
\end{equation*}
For the second piece, we apply H\"older's inequality and the Sobolev embedding $H^1 \hookrightarrow L^\infty$
\begin{align*}
\abs{\Re \inner{ib u_N^2 \partial_x^2 \overline{u}_N}{\partial_x^2 u_N}_{L^2}} &= \abs{b \Im \int u_N^2 (\partial_x^2 \overline{u}_N)^2 dx} \\
&\le \abs{b} \norm{u_N}_{L^\infty}^2 \norm{\partial_x^2 u_N}_{L^2}^2 \le C \norm{u_N}_{H^1}^2 \norm{\partial_x^2 u_N}_{L^2}^2.
\end{align*}
The third piece is bounded using the Gagliardo-Nirenberg inequality
\begin{align*}
\abs{\Re \inner{4ib u_N \abs{\partial_x u_N}^2}{\partial_x^2 u_N}_{L^2}} &\le 4\abs{b} \norm{u_N}_{L^\infty} \norm{\partial_x u_N}_{L^4}^2 \norm{\partial_x^2 u_N}_{L^2}& \\
& \le C \norm{u_N}_{H^1}^{5/2} \norm{\partial_x^2 u_N}_{L^2}^{3/2}.&
\end{align*}
The final piece is estimated symmetrically. Thus, $I_b \in \mathcal{G}$ since it can be globally bounded by
\begin{equation*}
\abs{I_b} \le C(\norm{u_N}_{H^1}) \left(1 + \norm{\partial_x^2 u_N}_{L^2}^2\right).
\end{equation*}

We continue estimating $I$ in \eqref{eq:I_term} by carefully isolating the highest-order terms arising from the second nonlinear contribution. Using integration by parts to shift one derivative, we have
\begin{align}
I_d &:= -d \Re \inner{\abs{u_N}^2 \partial_x u_N}{\partial_x^4 u_N}_{L^2} = d \Re \int \partial_x (\abs{u_N}^2 \partial_x u_N) (\partial_x^3 \overline{u}_N) dx \nonumber \\
&= d \Re \int \left[ \partial_x(\abs{u_N}^2) \partial_x u_N + \abs{u_N}^2 \partial_x^2 u_N \right] (\partial_x^3 \overline{u}_N) dx \nonumber \\
&= 2d \Re \int \Re(\overline{u}_N \partial_x u_N) (\partial_x u_N) (\partial_x^3 \overline{u}_N) dx + d \Re \int \abs{u_N}^2 (\partial_x^2 u_N) (\partial_x^3 \overline{u}_N) dx. \label{eq:Id_expansion}
\end{align}
The second integral in \eqref{eq:Id_expansion} provides exactly 
\begin{equation} \label{eq:Id_part2}
\begin{aligned}
d \Re \int \abs{u_N}^2 (\partial_x^2 u_N) (\partial_x^3 \overline{u}_N) dx &= \frac{d}{2} \int \abs{u_N}^2 \partial_x (\abs{\partial_x^2 u_N}^2) dx & \\
&= -d \int \Re(u_N \partial_x \overline{u}_N) \abs{\partial_x^2 u_N}^2 dx =- d T_1.&
\end{aligned}
\end{equation}
For the first integral in \eqref{eq:Id_expansion}, we integrate by parts once more to shift the derivative from $\partial_x^3 \overline{u}_N$ to the remaining factors
\begin{align}
&2d \Re \int \Re(\overline{u}_N \partial_x u_N) (\partial_x u_N) (\partial_x^3 \overline{u}_N) dx = -2d \Re \int \partial_x \left[ \Re(\overline{u}_N \partial_x u_N) \partial_x u_N \right] (\partial_x^2 \overline{u}_N) dx \nonumber \\
&= -2d \Re \int \left[ \Re(\abs{\partial_x u_N}^2 + \overline{u}_N \partial_x^2 u_N) \partial_x u_N + \Re(\overline{u}_N \partial_x u_N) \partial_x^2 u_N \right] (\partial_x^2 \overline{u}_N) dx \nonumber \\
&= -2d \int \abs{\partial_x u_N}^2 \Re(\partial_x u_N \partial_x^2 \overline{u}_N) dx - 2d \int \Re(\overline{u}_N \partial_x^2 u_N) \Re(\partial_x u_N \partial_x^2 \overline{u}_N) dx \nonumber \\
&\quad - 2d \int \Re(\overline{u}_N \partial_x u_N) \abs{\partial_x^2 u_N}^2 dx. \label{eq:Id_part1_expansion}
\end{align}
Notice that the first term in \eqref{eq:Id_part1_expansion} vanishes since 
\[
\abs{\partial_x u_N}^2 \Re(\partial_x u_N \partial_x^2 \overline{u}_N) = \frac{1}{4} \partial_x(\abs{\partial_x u_N}^4).
\]
The third term in \eqref{eq:Id_part1_expansion} is precisely $-2d T_1$.
For the middle term, we apply the complex identity $\Re(z_1)\Re(z_2) = \frac{1}{2}\Re(z_1 z_2) + \frac{1}{2}\Re(z_1 \overline{z}_2)$ with $z_1 = \overline{u}_N \partial_x^2 u_N$ and $z_2 = \partial_x u_N \partial_x^2 \overline{u}_N$:
\begin{align*}
-2d \int \Re(\overline{u}_N \partial_x^2 u_N) \Re(\partial_x u_N \partial_x^2 \overline{u}_N) dx &= -d \int \Re\left( u_N \partial_x u_N (\partial_x^2 \overline{u}_N)^2 \right) dx&  \\
& - d \int \Re(\overline{u}_N \partial_x u_N) \abs{\partial_x^2 u_N}^2 dx = -d T_2 -d T_1.&
\end{align*}
Summing all the components of \eqref{eq:Id_part1_expansion} with \eqref{eq:Id_part2}, we obtain the exact cancellation structure for $I_d$
\begin{equation} \label{eq:Id_final}
I_d = (-2d T_1 - d T_2 - d T_1) - d T_1 =- 4d T_1 - d T_2.
\end{equation}

Finally, we consider the third nonlinear contribution $I_e$ involving the conjugated derivative. Integrating by parts yields
\begin{align}
I_e &:= -e \Re \inner{u_N^2 \partial_x \overline{u}_N}{\partial_x^4 u_N}_{L^2} = e \Re \int \partial_x(u_N^2 \partial_x \overline{u}_N) (\partial_x^3 \overline{u}_N) dx \nonumber \\
&= 2e \Re \int u_N (\partial_x u_N) (\partial_x \overline{u}_N) (\partial_x^3 \overline{u}_N) dx + e \Re \int u_N^2 (\partial_x^2 \overline{u}_N) (\partial_x^3 \overline{u}_N) dx. \label{eq:Ie_expansion}
\end{align}
The second integral in \eqref{eq:Ie_expansion} directly generates a $T_2$ term
\begin{equation} \label{eq:Ie_part2}
e \int \Re\left(u_N^2 \frac{1}{2} \partial_x((\partial_x^2 \overline{u}_N)^2)\right) dx = -\frac{e}{2} \int \Re\left( 2 u_N \partial_x u_N (\partial_x^2 \overline{u}_N)^2 \right) dx = -e T_2.
\end{equation}
For the first integral in \eqref{eq:Ie_expansion}, we integrate by parts shifting the derivative to the left
\begin{align}
&2e \Re \int u_N \abs{\partial_x u_N}^2 (\partial_x^3 \overline{u}_N) dx = -2e \Re \int \partial_x\left( u_N \abs{\partial_x u_N}^2 \right) (\partial_x^2 \overline{u}_N) dx \nonumber \\
&= -2e \Re \int \left[ (\partial_x u_N) \abs{\partial_x u_N}^2 + u_N \partial_x(\abs{\partial_x u_N}^2) \right] (\partial_x^2 \overline{u}_N) dx \nonumber \\
&= -2e \int \abs{\partial_x u_N}^2 \Re(\partial_x u_N \partial_x^2 \overline{u}_N) dx - 4e \int \Re\left( u_N (\partial_x^2 \overline{u}_N) \Re(\partial_x u_N \partial_x^2 \overline{u}_N) \right) dx. \label{eq:Ie_part1_expansion}
\end{align}
Again, the first term vanishes as a total derivative. For the second term, applying the same complex identity trick as before on $\Re(u_N \partial_x^2 \overline{u}_N) \Re(\partial_x u_N \partial_x^2 \overline{u}_N)$, we find
\begin{align*}
-4e \int \Re\left( u_N \partial_x^2 \overline{u}_N \right) \Re\left( \partial_x u_N \partial_x^2 \overline{u}_N \right) dx &= -2e \int \Re(u_N \partial_x u_N (\partial_x^2 \overline{u}_N)^2) dx & \\
&- 2e \int \Re(u_N \partial_x \overline{u}_N) \abs{\partial_x^2 u_N}^2 dx= -2e T_2 - 2e T_1.&
\end{align*}
Combining this with \eqref{eq:Ie_part2}, we obtain the structure for $I_e$
\begin{equation} \label{eq:Ie_final}
I_e = (-2e T_2 - 2e T_1) - e T_2 = -2e T_1 - 3e T_2.
\end{equation}
Summing the contributions \eqref{eq:Id_final} and \eqref{eq:Ie_final}, and absorbing $I_b$ into the $\mathcal{G}$ remainder class, we rigorously establish \eqref{eq:evol_part1}.
\end{proof}

\begin{lemma} \label{lem:energy_part2}
Recall \eqref{C1C2}.
Let $N>1$ and $u_N \in C(\R, H^\infty)$ be the solution of \eqref{eq:approx_problem}. Then there exists a remainder $G_2 \in \mathcal{G}$, with the estimate \eqref{esti:G} uniform in $N$, such that
\begin{equation} \label{eq:evol_part2}
\frac{d}{dt} \mathcal{C}_1(u_N) =- 6c T_1 + G_2(u_N)
\end{equation}
with $T_1$ defined as in \eqref{eq:T1_def}.
\end{lemma}

\begin{proof}
We track the evolution of $\mathcal{C}_1(u_N)$. Its gradient action is
\begin{equation} \label{eq:G1_gradient}
\langle \mathcal{C}_1'(u_N), h \rangle = 2 \Re \int \abs{\partial_x u_N}^2 \overline{u}_N h dx + 2 \Re \int \abs{u_N}^2 (\partial_x \overline{u}_N) (\partial_x h) dx.
\end{equation}
We split the contributions using the functional derivative notation
\begin{equation} \label{eq:G1_split}
\frac{d}{dt} \mathcal{C}_1(u_N) = \langle \mathcal{C}_1'(u_N), -c \partial_x^3 u_N \rangle + \langle \mathcal{C}_1'(u_N), i a \partial_x^2 u_N \rangle + \langle \mathcal{C}_1'(u_N), J_N g(u_N) \rangle.
\end{equation}
To extract the highest-order bad terms, we only need to test the gradient against the third-order linear dispersion $-c \partial_x^3 u_N$, since we prove that all other flows will produce terms with at most 4 derivatives in $L^2$, which are bounded by $\norm{u_N}_{H^2}^2$ and thus belong to $\mathcal{G}$. For the first term in \eqref{eq:G1_split} we have 
\begin{equation} \label{eq:G1_dispersion}
\langle \mathcal{C}_1'(u_N), -c \partial_x^3 u_N \rangle = -2c \Re \int \abs{\partial_x u_N}^2 \overline{u}_N (\partial_x^3 u_N) dx - 2c \Re \int \abs{u_N}^2 (\partial_x \overline{u}_N) (\partial_x^4 u_N) dx.
\end{equation}
We integrate the first integral by parts
\begin{align*}
-2c \Re \int \abs{\partial_x u_N}^2 \overline{u}_N (\partial_x^3 u_N) dx &= 2c \Re \int \partial_x (\abs{\partial_x u_N}^2 \overline{u}_N) (\partial_x^2 u_N) dx \\
&= 4c \int \Re(\partial_x^2 u_N \partial_x \overline{u}_N) \Re(\overline{u}_N \partial_x^2 u_N) dx & \\
&+ 2c \int \abs{\partial_x u_N}^2 \Re(\partial_x \overline{u}_N \partial_x^2 u_N) dx.
\end{align*}
The second term is 
\[
c \int \partial_x(\abs{\partial_x u_N}^4)/2 = 0.
\]
For the first piece, applying the complex identity gives
\begin{equation} \label{eq:G1_part1}
4c \int \left[ \frac{1}{2}\Re(u_N \partial_x u_N (\partial_x^2 \overline{u}_N)^2) + \frac{1}{2}\Re(u_N \partial_x \overline{u}_N) \abs{\partial_x^2 u_N}^2 \right] dx = 2c T_2 + 2c T_1.
\end{equation}
For the second integral in \eqref{eq:G1_dispersion}, we apply integration by parts twice
\begin{align*}
&-2c \Re \int \abs{u_N}^2 (\partial_x \overline{u}_N) (\partial_x^4 u_N) dx = 2c \Re \int \partial_x(\abs{u_N}^2 \partial_x \overline{u}_N) (\partial_x^3 u_N) dx \\
&= 4c \int \Re(u_N \partial_x \overline{u}_N) \Re(\partial_x \overline{u}_N \partial_x^3 u_N) dx + 2c \int \abs{u_N}^2 \Re(\partial_x^2 \overline{u}_N \partial_x^3 u_N) dx.
\end{align*}
The second piece evaluates to 
\[
-c \int \partial_x(\abs{u_N}^2) \abs{\partial_x^2 u_N}^2 dx = -2c T_1.
\]
For the first piece, integrating by parts again yields
\begin{align}
&4c \Re \int \Re(u_N \partial_x \overline{u}_N) (\partial_x \overline{u}_N) (\partial_x^3 u_N) dx = -4c \Re \int \partial_x \left[ \Re(u_N \partial_x \overline{u}_N) \partial_x \overline{u}_N \right] (\partial_x^2 u_N) dx \nonumber \\
&= -4c \int \Re(u_N \partial_x^2 \overline{u}_N) \Re(\partial_x \overline{u}_N \partial_x^2 u_N) dx - 4c \int \Re(u_N \partial_x \overline{u}_N) \abs{\partial_x^2 u_N}^2 dx \nonumber \\
&= -4c \int \left[ \frac{1}{2}\Re(u_N \partial_x u_N (\partial_x^2 \overline{u}_N)^2) + \frac{1}{2}\Re(u_N \partial_x \overline{u}_N) \abs{\partial_x^2 u_N}^2 \right] dx - 4c T_1 \nonumber \\
&= -2c T_2 - 2c T_1 - 4c T_1 = -2c T_2 - 6c T_1. \label{eq:G1_part2}
\end{align}
Adding \eqref{eq:G1_part1} to \eqref{eq:G1_part2} and $-2c T_1$, we see the $T_2$ terms cancel
\begin{equation*}
\langle \mathcal{C}_1'(u_N), -c \partial_x^3 u_N \rangle = (2c T_2 + 2c T_1) + (-2c T_2 - 6c T_1) - 2c T_1 =- 6c T_1.
\end{equation*}
Now we need to prove that the other terms of \eqref{eq:G1_split} are in $\mathcal{G}$. We start with the linear term involving the coefficient $a$, which we omit. We estimate
\begin{align}
\langle \mathcal{C}_1'(u_N), i\partial_x^2 u_N \rangle &=- 2 \Im \int \abs{\partial_x u_N}^2 \overline{u}_N \partial_x^2 u_N dx - 2 \Im \int \abs{u_N}^2 (\partial_x \overline{u}_N) (\partial_x^3 u_N) dx \nonumber \\
&= - 2 \Im \int \abs{\partial_x u_N}^2 \overline{u}_N \partial_x^2 u_N dx + 2 \Im \int \partial_x (\abs{u_N}^2 \partial_x \overline{u}_N) (\partial_x^2 u_N) dx  \label{auxaux} \\
&= 2 \Im \int (\partial_x^2 u_N) \left[ \partial_x^2 \bar u_N |u_N|^2 + u_N (\partial_x \bar u_N)^2  \right] dx \nonumber \\
&=2 \Im \int u_N (\partial_x \bar u_N)^2 \partial_x^2 u_N  dx \nonumber,
\end{align}
where we use that 
\[
\Im \int \abs{u_N}^2 \abs{\partial_x^2 u_N}^2 dx = 0,
\]
and a part of the expansion of $\Re(\bar u_N \partial_x u_N) \partial_x \bar u_N$ cancels the first integral in \eqref{auxaux}.
Then, we have
\begin{align*}
\abs{ \int u_N (\partial_x \bar u_N)^2 \partial_x^2 u_N dx}& \le \norm{u_N}_{L^\infty} \norm{\partial_x u_N}_{L^4}^2 \norm{\partial_x^2 u_N}_{L^2} & \\
&\lesssim \norm{u_N}_{H^1}^{5/2} \norm{\partial_x^2 u_N}_{L^2}^{3/2}& \\
& \lesssim (1+\norm{u_N}_{H^2}^{2}),&
\end{align*}
by Gagliardo-Nirenberg and Sobolev embedding, where in the last inequality the omitted constant depends continuously on $\norm{u_N}_{H^1}$. This shows it belongs to $\mathcal{G}$.
The estimate of the evolution driven by the nonlinearity follows from direct Cauchy-Schwarz bounds. In particular, recalling the expression for the functional gradient from \eqref{eq:G1_gradient}, we test it against a generic function $h$
\begin{align} \label{eq:C1_gradient_expansion}
\langle  \mathcal{C}_1'(u_N), h \rangle &= 2\Re \int_{\mathbb{T}} |\partial_{x}u_N|^{2}\overline{u}_{N}h \, dx + 2\Re \int_{\mathbb{T}} |u_N|^{2}(\partial_{x}\overline{u}_{N})(\partial_{x}h) \, dx \nonumber \\
&= 2\Re \int_{\mathbb{T}} |\partial_{x}u_N|^{2}\overline{u}_{N}h \, dx - 2\Re \int_{\mathbb{T}} \partial_{x} \left( |u_N|^{2}\partial_{x}\overline{u}_{N} \right) h \, dx \nonumber \\
&= 2\Re \int_{\mathbb{T}} \mathcal{A}(u_N) h \, dx,
\end{align}
where, after explicitly expanding the spatial derivative in the second integral, we define the auxiliary term
\begin{equation} \label{eq:A_vN_def}
 \mathcal{A}(u_N) := |\partial_{x}u_N|^{2}\overline{u}_{N} - 2\Re(\overline{u}_{N}\partial_{x}u_N)\partial_{x}\overline{u}_{N} - |u_N|^{2}\partial_{x}^{2}\overline{u}_{N}.
\end{equation}
From this algebraic rewriting, it follows that if we test against the nonlinear flow $J_{N}g(u_N)$, we can bound it as
\begin{equation} \label{eq:C1_bound_initial}
\left| \langle  \mathcal{C}_1'(u_N), J_{N}g(u_N) \rangle \right| \lesssim \| \mathcal{A}(u_N)\|_{L^{2}} \|g(u_N)\|_{L^{2}}.
\end{equation}
Based on the definition of the operator $g$ in \eqref{eq:operators_def}, we can easily control the second factor. Since $L^{\infty}$ acts as an algebra multiplier on $L^{2}$, and using the embedding $H^{1} \hookrightarrow L^{\infty}$, we have
\begin{equation} \label{eq:g_vN_bound}
\|g(u_N)\|_{L^{2}} \lesssim \|u_N\|_{L^{\infty}}^{2} \|u_N\|_{H^{1}} \lesssim \|u_N\|_{H^{1}}^{3}.
\end{equation}

On the other hand, we must carefully estimate the contributions inside $A(u_N)$ in \eqref{eq:A_vN_def}. For the first term, by using the Gagliardo-Nirenberg interpolation inequality and standard Sobolev embeddings, we have
\begin{equation} \label{eq:A_term1_bound}
\int_{\mathbb{T}} |\partial_{x}u_N|^{4} |u_N|^{2} \, dx \le \|u_N\|_{L^{\infty}}^{2} \|\partial_{x}u_N\|_{L^{4}}^{4} \lesssim \|u_N\|_{H^{1}}^{5} \|u_N\|_{H^{2}}.
\end{equation}
The second term generated in $A(u_N)$ shares the exact same structural bounds. Finally, the third term involving the second derivative is estimated as
\begin{equation} \label{eq:A_term3_bound}
\int_{\mathbb{T}} |\partial_{x}^{2}u_N|^{2} |u_N|^{4} \, dx \le \|u_N\|_{L^{\infty}}^{4} \|\partial_{x}^{2}u_N\|_{L^{2}}^{2} \lesssim \|u_N\|_{H^{1}}^{4} \|u_N\|_{H^{2}}^{2}.
\end{equation}
By aggregating these bounds and taking the square root to recover the $L^{2}$ norm of $ \mathcal{A}(u_N)$, we deduce
\begin{equation} \label{eq:A_vN_final_bound}
\| \mathcal{A}(u_N)\|_{L^{2}} \le K_{1}(\|u_N\|_{H^{1}}) \left(1 + \|u_N\|_{H^{2}}\right),
\end{equation}
with $K_{1}$ depending continuously on the $H^{1}$ norm. Combining \eqref{eq:A_vN_final_bound} with \eqref{eq:g_vN_bound} via the Cauchy-Schwarz inequality in \eqref{eq:C1_bound_initial}, we definitively obtain
\begin{equation*}
\left| \langle  \mathcal{C}_1'(u_N), J_{N}g(u_N) \rangle \right| \le K_{2}(\|u_N\|_{H^{1}}) \left(1 + \|u_N\|_{H^{2}}\right).
\end{equation*}
Here $K_{2}$ also depends continuously on the lower-order norm, which demonstrates that $\langle  \mathcal{C}_1'(u_N), J_{N}g(u_N) \rangle$ belongs to the remainder class $\mathcal{G}$, precisely as required.

\end{proof}

\begin{lemma} \label{lem:energy_part3}
Recall \eqref{C1C2}.
Let $N>1$ and $u_N \in C(\R, H^\infty)$ be the solution of \eqref{eq:approx_problem}. Then there exists a remainder $G_3 \in \mathcal{G}$, with the estimate \eqref{esti:G} uniform in $N$, such that
\begin{equation} \label{eq:evol_part3}
\frac{d}{dt} \mathcal{C}_2(u_N) = -6c T_2 + G_3(u_N)
\end{equation}
with $T_2$ defined as in \eqref{eq:T2_def}.
\end{lemma}

\begin{proof}
As in the proof of Lemma \ref{lem:energy_part2}, we compute the first variation of \(\mathcal C_2\). For any smooth \(h\), we have
\begin{equation*}
\langle \mathcal{C}_2'(u_N), h \rangle = 2 \Re \int u_N (\partial_x \overline{u}_N)^2 h dx - 2 \Re \int \partial_x(u_N^2 \partial_x \overline{u}_N) \overline{h} dx.
\end{equation*}
The evolution in time is given by
\begin{equation*}
\frac{d}{dt} \mathcal{C}_2(u_N) = \langle \mathcal{C}_2'(u_N), -c \partial_x^3 u_N \rangle + \langle \mathcal{C}_2'(u_N), i a \partial_x^2 u_N \rangle + \langle \mathcal{C}_2'(u_N), J_N g(u_N) \rangle.
\end{equation*}
We track the evolution of $\mathcal{C}_2(u_N)$ along the linear third-order flow, as $\mathcal{C}_2$ also contains only two total derivatives and its remaining functional variations strictly fall into $\mathcal{G}$. We have
\begin{align} \label{eq:G2_dispersion}
\langle \mathcal{C}_2'(u_N), -c \partial_x^3 u_N \rangle &= -2c \Re \int u_N (\partial_x \overline{u}_N)^2 (\partial_x^3 u_N) dx + 2c \Re \int \partial_x(u_N^2 \partial_x \overline{u}_N) (\partial_x^3 \overline{u}_N) dx.
\end{align}
Notice that the second integral matches precisely the calculation we already performed for $I_e$ in Lemma \ref{lem:energy_part1} (see \eqref{eq:Ie_expansion}), just scaled by $2c$ instead of $e$. By extracting the result from \eqref{eq:Ie_final}, it evaluates directly to
\begin{equation} \label{eq:G2_part2}
2c \Re \int \partial_x(u_N^2 \partial_x \overline{u}_N) (\partial_x^3 \overline{u}_N) dx = 2c (-2 T_1 - 3 T_2) =- 4c T_1 - 6c T_2.
\end{equation}
For the first integral in \eqref{eq:G2_dispersion}, integration by parts gives
\begin{equation} \label{eq:G2_part1}
\begin{aligned}
-2c \Re \int u_N (\partial_x \overline{u}_N)^2 (\partial_x^3 u_N) dx &= 2c \Re \int \partial_x \left[ u_N (\partial_x \overline{u}_N)^2 \right] (\partial_x^2 u_N) dx  \\
&= 2c \int \abs{\partial_x u_N}^2 \Re(\partial_x \overline{u}_N \partial_x^2 u_N) dx& \\
& + 4c \int \Re(u_N \partial_x \overline{u}_N) \abs{\partial_x^2 u_N}^2 dx =  4c T_1.
\end{aligned}
\end{equation}
Adding \eqref{eq:G2_part2} and \eqref{eq:G2_part1} generates a perfect cancellation of the $T_1$ cross-terms
\begin{equation}\label{aux:11}
\langle \mathcal{C}_2'(u_N), -c \partial_x^3 u_N \rangle = (-4c T_1 - 6c T_2) + 4c T_1 = -6c T_2,
\end{equation}
which provides the desired identity \eqref{eq:evol_part3}.
The estimates of the other terms follow from the same estimates used above to prove \eqref{eq:C1_bound_initial}.
\end{proof}

\begin{proof}[Proof of Theorem \ref{thm:energy}]
The proof follows directly from the previous lemmas. We consider the time derivative of the modified energy defined in \eqref{eq:modified_energy_def}. Combining algebraic results from Lemmas \ref{lem:energy_part1}, \ref{lem:energy_part2}, and \ref{lem:energy_part3}, we have
\begin{align*}
\frac{d}{dt} \mathcal{E}_N(u_N) &= \frac{d}{dt} \norm{\partial_x^2 u_N}_{L^2}^2 - \frac{4d+2e}{3c} \frac{d}{dt} \mathcal{C}_1(u_N) - \frac{d+3e}{3c} \frac{d}{dt} \mathcal{C}_2(u_N) \\
&=- 2(4d+2e)T_1 - 2(d+3e)T_2 + \frac{4d+2e}{3c}(6c T_1) + \frac{d+3e}{3c}(6c T_2) + G_{rem,N}(u_N) \\
&= \left[ -2(4d+2e) + 2(4d+2e) \right] T_1 + \left[ -2(d+3e) + 2(d+3e) \right] T_2 + G_{rem,N}(u_N),\\
&= G_{rem,N}(u_N).
\end{align*}

where
\[
G_{\mathrm{rem},N}
:=2G_1-\frac{4d+2e}{3c}G_2
-\frac{d+3e}{3c}G_3.
\]
Since each $G_i$ belongs to $\mathcal G$ with
bounds uniform in $N$, the same holds for
$G_{\mathrm{rem},N}$, completing the proof.
\end{proof}
\subsection{Proof of Theorem \ref{GWPG}}

We already proved the local well-posedness (LWP) result for $s>3/2$ and $u_0 \in H^s$ in Lemma \ref{lem:local_existence_true_PDE}. We need to prove the conditional global well-posedness (GWP) by applying Theorem \ref{thm:energy}. Let $u \in C((-T_{1,\textnormal{max}}, T_{2,\textnormal{max}}), H^2)$ be the maximal solution of \eqref{eq:main_pde} for some fixed $u_0 \in H^2$. We are going to prove the second item of Theorem  \ref{GWPG}, making the contradictory assumption that
\begin{equation} \label{eq:blowup_contradiction}
T_{2,\textnormal{max}} < +\infty, \quad \sup_{t \in [0 , T_{2,\textnormal{max}})} \norm{u(t)}_{H^1} < +\infty.
\end{equation}

\medskip

We need the following result in order to apply the energy estimate from Theorem \ref{thm:energy}.

\begin{lemma} \label{lem:uniform_H1}
Let $u \in C([0,T_{2,\textnormal{max}}), H^2)$ be the maximal solution of the Cauchy problem \eqref{eq:main_pde} for positive times, satisfying \eqref{eq:blowup_contradiction}. 
 For each $N\geq1$, let $u_N\in C(\mathbb R;H^\infty)$ be the solution of the approximate problem \eqref{eq:approx_problem} with initial datum $J_N u_0$.\\
Then there exists $M_* > 0$ such that for any $T_2 \in (0,T_{2,\textnormal{max}})$ there exists $N_1 > 1$ satisfying
\begin{equation} \label{eq:uniform_H1_lemma}
\sup_{N \ge N_1} \sup_{t \in [0, T_2]} \norm{u_N(t)}_{H^1} \le M_*.
\end{equation}
\end{lemma}
The proof follows arguments by \cite{tsutsumi1980solutions, hayashi2025global}.

\begin{remark}
The key point of the above lemma is to provide a uniform $H^1$ bound
of the approximate solutions to any compact time interval
$[0,T_2]\subset [0,T_{2,\max})$.
Indeed, by applying directly the $H^1$-strong convergence $u_N \to u$, we could obtain this bound only on the local time of existence of $u_N$ given in Lemma \ref{UNintermediate}. 
\end{remark}

{\begin{proof}
We consider $s\in (3/2, 2)$.
Since $u_0\in H^2\subset H^s$, by Lemma \ref{UNintermediate} for any $M>0$ such that
\begin{equation}\label{M}
\norm{u_N(0)}_{H^s} \le M \qquad \forall N\geq 1,
\end{equation}
there exist $T_0(M) > 0$ and $C_{1,M} > 0$ such that
\begin{equation} \label{eq:local_Hs_bound}
\sup_{N \ge 1} \sup_{t \in [0, T_0(M)]} \norm{u_N(t)}_{H^s} \le C_{1,M}\,.
\end{equation}
Moreover, by the same computations as in Lemma \ref{lem:local_existence_true_PDE}, by bounding the $W^{1,\infty}$ norm with the $H^s$ norm in \eqref{eq:diff_ineq_Hs}, we obtain that
\begin{equation}\label{eq:local_Hs_bound:Gro}
\frac{d}{dt} \norm{u_N(t)}_{H^2}^2 \le C(\norm{u_N(t)}_{H^s}) \norm{u_N(t)}_{H^2}^2,
\end{equation}
with $C>0$ a constant depending continuously on the $H^s$ norm. 
By \eqref{eq:local_Hs_bound} and \eqref{eq:local_Hs_bound:Gro}, there exists $C_{2,M} > 0$ such that
\begin{equation} \label{eq:local_H2_bound}
\norm{u_N(t)}_{H^2} \le \norm{u_N(0)}_{H^2} \exp(C_{2,M} t), \quad \forall t \in [0,T_0(M)].
\end{equation}
By the contradictory assumption,
\begin{equation}\label{Mstar}
M_* = \sup_{t \in [0, T_{2,\textnormal{max}})} \norm{u(t)}_{H^1} + 1,
\end{equation}
is finite.
For any $T_2 \in (0, T_{2,\textnormal{max}})$ we fix 
\begin{equation*}
M = \sup_{t \in [0, T_2]} \norm{u(t)}_{H^s} + 1,
\end{equation*}
which satisfies \eqref{M}. {Then \eqref{eq:local_H2_bound} is satisfied for this choice of $M$}.
By the strong convergence of $u_N \to u$ in $C([0, T_0], H^{s})$ obtained in Lemma \ref{UNintermediate2} (applied to $H^2$ and $s<2$), there exists $N_0 > 1$ such that for all $N \ge N_0$
\begin{equation*}
\norm{u_N(T_0)}_{H^s} \le \norm{u(T_0)}_{H^s} + 1 \le M.
\end{equation*}
Then we can apply the same argument and \eqref{eq:local_Hs_bound} with initial datum $u_N(T_0)$. Iterating the estimates for at most $[T_2/T_0] + 1$ times, we have $u_N \to u$ in $C([0, T_2], H^{s'})$ for any $s' <2$. Then by the convergences and \eqref{Mstar} it holds
\begin{equation*}
\sup_{N \ge N_1} \sup_{t \in [0, T_2]} \norm{u_N(t)}_{H^1} \le M_*.
\end{equation*}
\end{proof}}

\begin{proof}[Proof of Theorem \ref{GWPG}]
We start the proof by considering the second assertion of the Theorem \ref{GWPG} for $s=2$, since the proof for the negative times will follow by similar computations. Suppose by contradiction \eqref{eq:blowup_contradiction}, then by Lemma \ref{lem:uniform_H1}, there exists a constant $M_* > 0$ (depending possibly on the initial datum) such that, for any $T \in (0, T_{2, \textnormal{max}})$, there exists an integer $N_1 =N_1(T)> 1$ such that
\begin{equation} \label{eq:H1_uniform_bound}
\sup_{N \ge N_1} \sup_{t \in [0, T]} \norm{u_N(t)}_{H^1} \le M_*.
\end{equation}
Under this assumption, we want to prove that the $H^2$ norm of $u_N$ cannot blow up, which in turn contradicts \eqref{eq:blowup_contradiction} and implies $T_{2,\textnormal{max}} = +\infty$, thus establishing global existence.

We consider the modified energy $\mathcal{E}_N(u_N(t))$ constructed in Theorem \ref{thm:energy}. Recalling its definition \eqref{eq:modified_energy_def}, we can rewrite it as
\begin{equation*}
\norm{\partial_x^2 u_N(t)}_{L^2}^2 = \mathcal{E}_N(u_N(t)) + \frac{4d+2e}{3c} \mathcal{C}_1(u_N(t)) + \frac{d+3e}{3c} \mathcal{C}_2(u_N(t)).
\end{equation*}
Since the functionals $\mathcal{C}_1$ and $\mathcal{C}_2$ contain at most first-order derivatives, they are bounded by constants which depend continuously on the $H^1$ norm. Specifically, by the bounds 
\[
|\mathcal{C}_1(u_N)|, |\mathcal{C}_2(u_N)|\le \| u_N \|_{L^{\infty}}^2 \| u_N \|_{H^1}^2, 
\]
combined with the assumption \eqref{eq:H1_uniform_bound}, we have that there exists a constant $K_1 = K_1(M_*) > 0$ such that for all $t \in [0, T]$ and for all $N \ge N_1$
\begin{equation} \label{eq:G_bound}
\abs{\frac{4d+2e}{3c} \mathcal{C}_1(u_N(t)) + \frac{d+3e}{3c} \mathcal{C}_2(u_N(t))} \le K_1(M_*).
\end{equation}
This establishes the equivalence between the modified energy and the $L^2$-norm of the second derivative up to a bounded error
\begin{equation} \label{eq:energy_equivalence}
\abs{\norm{\partial_x^2 u_N(t)}_{L^2}^2 - \mathcal{E}_N(u_N(t))} \le K_1(M_*), \quad \forall t \in [0, T].
\end{equation}
In particular, evaluating \eqref{eq:energy_equivalence} at time $t=0$, we can bound the initial modified energy
\begin{equation} \label{eq:initial_energy_bound}
\mathcal{E}_N(u_N(0)) \le \norm{\partial_x^2 J_N u_0}_{L^2}^2 + K_1(M_*) \le \norm{u_0}_{H^2}^2 + K_1(M_*).
\end{equation}

Now, we use the differential identity \eqref{eq:modified_energy_evol} from Theorem \ref{thm:energy}. We recall that $G_{rem,N} \in \mathcal{G}$ and the corresponding constant in the $\mathcal G$-bound is uniform with respect to $N$. Combining that property with \eqref{eq:H1_uniform_bound}, the definition of the class $\mathcal{G}$ guarantees the existence of a constant $K_2 = K_2(M_*) > 0$ such that
\begin{equation*}
\left|\frac{d}{dt} \mathcal{E}_N(u_N(t))\right| = \abs{G_{rem,N}(u_N(t))} \le K_2(M_*) \left( 1 + \norm{u_N(t)}_{H^2}^2 \right).
\end{equation*}
Integrating this inequality over the time interval $[0, t]$ for any $t \le T$, we obtain
\begin{equation} \label{eq:integrated_energy}
\mathcal{E}_N(u_N(t)) \le \mathcal{E}_N(u_N(0)) + K_2(M_*) \int_0^t \left( 1 + \norm{u_N(\tau)}_{H^2}^2 \right) d\tau.
\end{equation}
To close the estimate, we need to reconstruct the full $H^2$ norm on the left-hand side. We use the standard norm equivalence $\norm{u_N}_{H^2}^2 \sim \norm{u_N}_{L^2}^2 + \norm{\partial_x^2 u_N}_{L^2}^2$ and the exact mass conservation from Lemma \ref{lem:global_approx}, which states that $\norm{u_N(t)}_{L^2}^2 = \norm{J_N u_0}_{L^2}^2 \le \norm{u_0}_{H^2}^2$. 

Combining the lower bound from \eqref{eq:energy_equivalence}, the initial bound \eqref{eq:initial_energy_bound}, and the integrated inequality \eqref{eq:integrated_energy}, we have
\begin{align} \label{eq:gronwall_setup}
\norm{u_N(t)}_{H^2}^2 &\leq 2\norm{u_N(t)}_{L^2}^2 + 2\norm{\partial_x^2 u_N(t)}_{L^2}^2 \nonumber \\
&\leq 2\norm{u_0}_{H^2}^2 + 2\left( \mathcal{E}_N(u_N(t)) + K_1(M_*) \right) \nonumber \\
&\leq 2\norm{u_0}_{H^2}^2 +2 K_1(M_*) + 2\mathcal{E}_N(u_N(0)) +2 K_2(M_*) \int_0^t \left( 1 + \norm{u_N(\tau)}_{H^2}^2 \right) d\tau \nonumber \\
&\le 4\norm{u_0}_{H^2}^2 + 4K_1(M_*) + 2K_2(M_*) \int_0^t \left( 1 + \norm{u_N(\tau)}_{H^2}^2 \right) d\tau.
\end{align}
Let us define the constants $\alpha = 1 + 4\norm{u_0}_{H^2}^2 + 4K_1(M_*)$ and $\beta = 2K_2(M_*)$. Then from \eqref{eq:gronwall_setup}, we obtain the integral inequality
\begin{equation*}
1 + \norm{u_N(t)}_{H^2}^2 \le \alpha + \beta \int_0^t \left( 1 + \norm{u_N(\tau)}_{H^2}^2 \right) d\tau.
\end{equation*}
Applying Gronwall's lemma, we deduce that for all $t \in [0, T]$ and uniformly for $N \ge N_1$
\begin{equation}\label{stimun}
\norm{u_N(t)}_{H^2}^2 \le \alpha e^{\beta t}.
\end{equation}
Reasoning as in the proof of Lemma \ref{UNintermediate2}, since $\alpha$ and $\beta$ are uniform in $N$, we obtain 
\begin{equation} \label{eq:final_gronwall_bound}
\norm{u(t)}_{H^2}^2 \le \left( 1 + 4\norm{u_0}_{H^2}^2 + 4K_1(M_*) \right) e^{2K_2(M_*) T_{2,\max}}.
\end{equation}
Since the right-hand side is a finite constant depending only on the initial datum, the time $T$, and the $H^1$ bound $M_*$, the $H^2$ norm of the solution cannot blow up. This implies that we can extend the existence of the solution for some $T_*>T_{2,\text{max}}$, which is a contradiction since $T_{2,\text{max}}$ is the maximal time of existence. Thus, for those solutions whose $H^1$ norm is bounded,  $T_{2,\text{max}}=+\infty$, and the proof of the blow-up alternative for positive times is complete for $s=2$. 

The proof of the third item for $s=2$ in Theorem \ref{GWPG} follows a similar strategy.

\medskip

{The case $u \in C((-T_{1,\textnormal{max}},T_{2,\textnormal{max}}),H^s)$ with $s>2$ follows from the case $s=2$. Indeed, assume by contradiction that
\[
T_{2,\textnormal{max}}<+\infty,
\qquad
\sup_{t\in[0,T_{2,\textnormal{max}})}
\norm{u(t)}_{H^1}<+\infty.
\]
Let $v\in C((-S_{1,\textnormal{max}},S_{2,\textnormal{max}}),H^2)$ be the maximal $H^2$-solution to \eqref{eq:main_pde} with the same initial datum $v(0)=u_0$. Such a solution exists since $u_0\in H^s\subset H^2$. By the uniqueness established in Lemma \ref{lem:local_existence_true_PDE}, the two solutions coincide on their common interval of existence,
\[
v(t)=u(t),
\qquad
t\in
[0,\min\{T_{2,\textnormal{max}},S_{2,\textnormal{max}}\}).
\]
We claim that $S_{2,\textnormal{max}}>T_{2,\textnormal{max}}$. Indeed, otherwise $S_{2,\textnormal{max}}\leq T_{2,\textnormal{max}}<+\infty$, and the identity $v=u$ on $[0,S_{2,\textnormal{max}})$ would imply
\[
\sup_{t\in[0,S_{2,\textnormal{max}})}
\norm{v(t)}_{H^1}
\leq
\sup_{t\in[0,T_{2,\textnormal{max}})}
\norm{u(t)}_{H^1}
<+\infty.
\]
This contradicts the blow-up alternative already proved in the case $s=2$, according to which a finite maximal $H^2$ existence time necessarily implies blow-up of the $H^1$ norm. Hence $S_{2,\textnormal{max}}>T_{2,\textnormal{max}}$.
Consequently, the $H^2$-solution $v$ is defined on an interval strictly larger than $[0,T_{2,\textnormal{max}}]$ and then there exists a finite constant
\[
M:=\sup_{t\in[0,T_{2,\textnormal{max}}]}
\norm{v(t)}_{H^2}<+\infty.
\]
Since $u=v$ on $[0,T_{2,\textnormal{max}})$, we therefore have
\begin{equation}\label{M:H2}
\sup_{t\in[0,T_{2,\textnormal{max}})}
\norm{u(t)}_{H^2}\leq M.
\end{equation}
Recalling the first inequality in \eqref{eq:diff_ineq_Hs} and using the above bound, by Gronwall's inequality,
\[
\norm{u(t)}_{H^s}^2 \leq \norm{u_0}_{H^s}^2 \exp\!\left(C(M)t\right) \leq \norm{u_0}_{H^s}^2 \exp\!\left(C(M)T_{2,\textnormal{max}}\right), \qquad t\in[0,T_{2,\textnormal{max}}),
\]
for some $C(M)>0$ depending only on $M$ and on the coefficients of the equation. Therefore,
\[
\sup_{t\in[0,T_{2,\textnormal{max}})} \norm{u(t)}_{H^s}<+\infty.
\]
This uniform $H^s$ bound contradicts the maximality of $T_{2,\textnormal{max}}$.  The proof for negative times is completely analogous, and the blow-up alternative is therefore proved for every $s>2$.}
\end{proof}

\subsection{Global Well-Posedness.}

We are now able to prove Corollary \ref{e0GWP}.

\begin{proof}[Proof of Corollary \ref{e0GWP}]
Fix $s\geq 2$ and an initial datum $u_0 \in H^s$, then the local existence follows from Lemma \ref{lem:local_existence_true_PDE}. We assume that the coefficients in \eqref{eq:main_pde} satisfy the condition 
\begin{equation}\label{condition}
abc\neq 0, \qquad e \neq 0,
\end{equation}
then we need to prove that the solution extends globally, i.e., $u \in C(\mathbb{R}, H^s)$. We proceed by contradiction. Let $u \in C([0, T_{\text{max}}), H^s)$ be the maximal solution (for positive times), and assume $T_{\text{max}} < +\infty$. By the blow-up alternative established in Theorem \ref{GWPG}, this implies
\begin{equation} \label{eq:blowup_H1_assumption}
\limsup_{t \nearrow T_{\text{max}}} \norm{u(t)}_{H^1} = +\infty.
\end{equation}

To study this possible blow-up, we introduce the modified energy
\begin{equation} \label{eq:modified_energy_H1}
I(u) := \int_{\mathbb{T}} |u_x|^2 dx - \frac{d+e}{6c} \int_{\mathbb{T}} |u|^4 dx + \frac{3bc-a(d+e)}{3ce} \Im \int_{\mathbb{T}} \overline{u} u_x dx,
\end{equation}
which is well-defined since $ec \neq 0$. By the Gagliardo-Nirenberg and Cauchy-Schwarz inequalities, we have
\[
\norm{u}_{L^4}^4 \lesssim \norm{u}_{L^2}^3 \norm{u}_{H^1}, \qquad \abs{\Im \int \overline{u} u_x dx} \le \norm{u}_{L^2} \norm{u_x}_{L^2}.
\]
Thus, for any $v \in H^1$, we have
\begin{equation} \label{eq:energy_equivalence_H1}
\abs{ I(v) - \norm{\partial_x v}_{L^2}^2 } \le C(\|v\|_{L^2}) \left( 1 + \norm{v}_{H^1} \right),
\end{equation}
with $C>0$ depending continuously on the $L^2$ norm.

Consider now $u_N \in C(\mathbb{R}, H^\infty)$ the solution of the approximate Cauchy problem \eqref{eq:approx_problem}. We now explicitly evaluate the time derivative of $I(u_N(t))$ on approximate solutions $\partial_t u_N = A u_N + J_N g(u_N)$, where $A$ and $g$ are as in \eqref{eq:operators_def}. We set $q := |u_N|^2$ and define
\begin{equation*}
K(u_N) := \Im \int_{\mathbb{T}} u_N^2 \overline{\partial_x u_N}^2 dx, \qquad M(u_N) := \int_{\mathbb{T}} \partial_x q |\partial_x u_N|^2 dx.
\end{equation*}

We start by computing 
\[
\frac{d}{dt} \int |\partial_x u_N|^2 dx = -2 \Re \int \partial_t u_N \overline{\partial_{xx} u_N} dx.
\]
Since $u_N$ is frequency-localized, the projection $J_N$ acts as the identity on $\overline{\partial_{xx} u_N}$. Substituting the PDE terms we first observe that 
\[
-2 \Re \int i a \partial_{xx} u_N \overline{\partial_{xx} u_N} dx = 0.
\]
Then for the term involving the coefficient $b$ we have
\[
2b \Im \int q u_N \overline{\partial_{xx} u_N} dx = -2b \Im \int \partial_x(q u_N) \overline{\partial_x u_N} dx.
\]
Since 
\[
\partial_x q u_N \overline{\partial_x u_N} = q |\partial_x u_N|^2 + u_N^2 \overline{\partial_x u_N}^2,
\]
taking the imaginary part leaves exactly $-2bK$.
For the term involving the coefficient $c$ we have 
\[
2c \Re \int \partial_{xxx} u_N \overline{\partial_{xx} u_N} dx = c \int \partial_x |\partial_{xx} u_N|^2 dx = 0.
\]
The part involving the coefficient $d$ satisfies 
\[
2d \Re \int q \partial_x u_N \overline{\partial_{xx} u_N} dx = d \int q \partial_x |\partial_x u_N|^2 dx = -d \int \partial_x q |\partial_x u_N|^2 dx = -dM(u_N).
\]
Finally, the term involving the coefficient $e$ is 
\[
2e \Re \int u_N^2 \overline{\partial_x u_N} \overline{\partial_{xx} u_N} dx = e \Re \int u_N^2 \partial_x(\overline{\partial_x u_N}^2) dx = -2e \Re \int u_N \partial_x u_N \overline{\partial_x u_N}^2 dx.
\]
Since $M(u_N)= 2 \Re \int u_N \partial_x u_N \overline{\partial_x u_N}^2 dx$, this yields $-eM(u_N)$.
Summing these gives:
\begin{equation} \label{eq:deriv_kinetic}
\frac{d}{dt} \int_{\mathbb{T}} |\partial_x u_N|^2 dx = -2b K(u_N) - (d+e) M(u_N).
\end{equation}

We now compute $\frac{d}{dt} \Im \int \overline{u_N} \partial_x u_N dx = -2 \Im \int \partial_t u_N \overline{\partial_x u_N} dx$. As before, the self-adjointness of $J_N$ and the identity $J_N\partial_x u_N=\partial_x u_N$ allow us to remove the projection. The term involving the coefficient $a$ is again zero. Moreover
\begin{align*}
-2b \Re \int q u_N \overline{\partial_x u_N} dx = -b \int q \partial_x q dx = 0, \\
2c \Im \int \partial_{xxx} u_N \overline{\partial_x u_N} dx = -2c \Im \int |\partial_{xx} u_N|^2 dx = 0, \\
2d \Im \int q |\partial_x u_N|^2 dx = 0. 
\end{align*}
Finally
\[
-2 \Im \int (-e u_N^2 \overline{\partial_x u_N}) \overline{\partial_x u_N} dx = 2e \Im \int u_N^2 \overline{\partial_x u_N}^2 dx = 2eK(u_N).
\]
Thus
\begin{equation} \label{eq:deriv_momentum}
\frac{d}{dt} \Im \int_{\mathbb{T}} \overline{u_N} \partial_x u_N dx = 2e K(u_N).
\end{equation}

Now we consider $\frac{d}{dt} \int q^2 dx = 4 \Re \int q \partial_t u_N \overline{u_N} dx$. Here, the test function $q \overline{u_N}$ has frequencies up to $3N$, generating a projection remainder. By computations similar to those introduced above we have
\begin{equation} \label{eq:deriv_L4}
\frac{d}{dt} \int_{\mathbb{T}} |u_N|^4 dx = -4a K - 6c M + 4 \Re \int_{\mathbb{T}} |u_N|^2 u_N \overline{(J_N - I) g(u_N)} dx.
\end{equation}
Combining \eqref{eq:deriv_kinetic}, \eqref{eq:deriv_momentum}, and \eqref{eq:deriv_L4} into the time derivative of $I(u_N)$, we verify the exact cancellation of $M(u_N)$ and $K(u_N)$. 
For the term involving $M(u_N)$
\[
-(d+e)M(u_N) - \frac{d+e}{6c}(-6cM(u_N)) = 0.
\]
For the one involving $K(u_N)$
\begin{equation}\label{Term:K}
-2bK(u_N) - \frac{d+e}{6c}(-4aK(u_N)) + \frac{3bc-a(d+e)}{3ce}(2eK(u_N))  = 0.
\end{equation}
The principal terms cancel, leaving only the projection remainder $\tilde{R}_N(t)$:
\begin{equation} \label{eq:I_deriv_uN}
\frac{d}{dt} I(u_N(t)) = \tilde{R}_N(t) := \frac{d+e}{6c} 4 \Re \int_{\mathbb{T}} \overline{g(u_N)} (I - J_N) (|u_N|^2 u_N) dx.
\end{equation}

By leveraging the properties of the spectral projector $J_N$ and the fractional Leibniz rule, this error is bounded by
\begin{align}
\abs{\tilde{R}_N(t)} &\lesssim \norm{g(u_N)}_{L^2} \norm{(I - J_N)(\abs{u_N}^2 u_N)}_{L^2} \nonumber \\
&\lesssim \left( \norm{u_N}_{L^\infty}^2 \norm{u_N}_{L^2} +\norm{u_N}_{L^\infty}^2 \norm{\partial_x u_N}_{L^2} \right) \frac{1}{N}  \left( \|\abs{u_N}^2 u_N\|_{H^1} \right) \nonumber \\
&\lesssim \frac{1}{N} \norm{u_N}_{H^1}^6, \label{eq:RN_bound}
\end{align}
from which 
\begin{equation} \label{eq:dt_I_closed}
\left| \frac{d}{dt} I(u_N(t)) \right| \lesssim \frac{1}{N} \left(1 + \|u_N\|_{H^1}^6\right).
\end{equation}
Moreover, by \eqref{eq:energy_equivalence_H1} and Lemma \ref{lem:JN_properties} we have
\begin{align} \label{eq:H1_equiv_explicit}
I(u_{N}(0)) \lesssim \left(1 + \|u_{0}\|_{H^{1}}^{2}\right),
\end{align}
where the constants are uniform in $N$ and depend continuously on $\|u_0\|_{L^2}$.

By integrating \eqref{eq:dt_I_closed} over time for any $t \in [0,T_{\mathrm{max}})$, we obtain
\begin{equation*}
I(u_{N}(t))- I(u_{N}(0))
\lesssim
 \frac{1}{N} \int_{0}^{t}
\left(1+\|u_{N}(\tau)\|_{H^{1}}^{6}\right)d\tau.
\end{equation*}
We now aim to reconstruct the $H^{1}$ norm on the left-hand side. By \eqref{eq:energy_equivalence_H1}, \eqref{conservaL2}  and \eqref{eq:H1_equiv_explicit}, we have
\begin{align} \label{eq:H1_reconstruction}
\|u_{N}(t)\|_{H^{1}}^{2}
&\leq \|u_{N}(0)\|_{L^{2}}^{2}
+ \|\partial_{x}u_{N}(t)\|_{L^{2}}^{2} \nonumber \\
&\le \|u_{N}(0)\|_{L^{2}}^{2}
+ I(u_{N}(t))
+C(\| u_0\|_{L^2}) \left(1+\|u_{N}(t)\|_{H^{1}}\right) \nonumber \\
&\lesssim \left(1+\|u_{0}\|_{H^{1}}^{2}\right)
+ \left(1+\|u_{N}(t)\|_{H^{1}}\right) \nonumber \\
&\quad
+ \frac{1}{N} \int_{0}^{t}
\left(1+\|u_{N}(\tau)\|_{H^{1}}^{6}\right)d\tau,
\end{align}
where the omitted constants are uniform in $N$ and depend continuously on $\|u_0\|_{L^2}$. Let us introduce the auxiliary variable $X(t):=1+\|u_{N}(t)\|_{H^{1}}^{2}.$ By Young's inequality, for any $C_0>1$ there exists $C_1>0$ such that
\begin{equation*}
C_{0}X(t)^{1/2}
\leq \frac{1}{2}X(t)+C_{1}.
\end{equation*}
Therefore, after absorbing the lower-order term on the left-hand side, we obtain
\begin{equation} \label{eq:X_integral_ineq}
X(t)
\leq C_{2}
+ \frac{1}{N} C_{3}\int_{0}^{t}X(\tau)^{3}d\tau
=:G(t),
\end{equation}
for appropriate constants $C_{2},C_{3}>0$ uniform in $N$.
Differentiating the upper bound $G(t)$, we obtain
\begin{equation} \label{eq:G_ode_general}
\frac{d}{dt}G(t)
\leq \frac{C_3}{N} G(t)^{3},
\end{equation}
with the initial condition $G(0)=C_{2}.$ Here, $C_{2}>0$ depends continuously on $\|u_{0}\|_{H^{1}}$, but is independent of $N$.
From \eqref{eq:G_ode_general} we obtain
\begin{equation} \label{eq:G_solution_general}
G(t)
\leq
\frac{C_{2}}
{\left(
1-2\frac{1}{N} C_{4}t
\right)^{\frac{1}{2}}},
\end{equation}
for some constant $C_{4}>0$ uniform in $N$.
Let $\widetilde{K}_{1}>\sqrt{C_{2}}$ be fixed. By setting 
\[
\widetilde{K}_{2}:= \frac{1-\left(\frac{C_2}{\widetilde{K}_1^2}\right)^2} {2C_4}>0,
\]
for any $ t \leq \widetilde{K}_{2}N$ we have
\begin{equation*}
G(t)\leq \widetilde{K}_{1}^2.
\end{equation*}
Consequently, for any $N\geq 1$ and any $t\in[0,\widetilde{K}_{2}N]$,
\begin{equation}\label{eq:safe_timespan}
\|u_{N}(t)\|_{H^{1}}^{2}
\leq G(t)
\leq \widetilde{K}_{1}^2.
\end{equation}

We now force the contradiction. Since $T_{\text{max}} < +\infty$, by \eqref{eq:blowup_H1_assumption} we can fix $t^* < T_{\text{max}}$ such that $\norm{u(t^*)}_{H^1} > 2 \widetilde{K}_{1}$. Because $t^*$ is strictly smaller than $T_{\text{max}}$, $u \in C([0, t^*], H^s)$. Since $u_N \to u$ strongly in $H^1$ uniformly on $[0, t^*]$, there exists an index $N_0 \ge 1$ such that for all $N \ge N_0$,
\begin{equation} \label{eq:contradiction_lower}
\norm{u_N(t^*)}_{H^1} > \widetilde{K}_{1}.
\end{equation}
As $N \to +\infty$, the interval $[0,  \widetilde{K}_{2} N]$ eventually covers $t^*$. Choosing $N_1 \ge N_0$ large enough such that $ \widetilde{K}_{2} N_1 > t^*$, the established bound \eqref{eq:safe_timespan} implies $\norm{u_{N_1}(t^*)}_{H^1} \le \widetilde{K}_{1}$. This contradicts \eqref{eq:contradiction_lower}. Therefore, $T_{\text{max}} = +\infty$, and the solution extends globally.  
\end{proof}

\begin{remark}\label{rem:Hirota} The global well-posedness result of Corollary \ref{e0GWP} extends to the case $e=0$ and $d=3bc/a$, corresponding to the integrable Hirota equation \cite{hirota1973exact}. Indeed, in this case the functional
\[
\int_{\mathbb T}|u_x|^2\,dx-\frac{b}{2a}\int_{\mathbb T}|u|^4\,dx
\]
is conserved along smooth solutions, as follows from the cancellations in \eqref{eq:deriv_kinetic}-\eqref{Term:K}. Together with mass conservation, this provides a uniform $H^1$ bound. The argument used to prove Corollary \ref{e0GWP} then yields global well-posedness in $H^s(\mathbb T)$ for every $s\geq2$.
\end{remark}

\section{The Perturbed NLS-Dysthe Equations}\label{per:DNLS}

In this section, we establish long-time uniform (in $\eps$) bounds for the perturbative problem \eqref{app:D}. To go beyond the standard $\mathcal{O}(1)$ existence time and reach time scales of order $\mathcal{O}(\eps^{-1})$, we exploit the conservation laws of the NLS part. Since these quantities are conserved by the unperturbed flow, their variation along the perturbed dynamics is of order $\eps$.

{It is well known that the unperturbed cubic NLS equation
\begin{equation}\label{NLS}
i \partial_t u + a \partial_x^2 u + b |u|^2 u = 0,
\end{equation}
admits an infinite hierarchy of conserved quantities
$\{E_n(u)\}_{n\in\mathbb N}$. We briefly recall their inductive
construction. Following \cite{kenig2025well}, consider first the normalized
cubic NLS
\begin{equation}\label{eq:integrable_nls}
i\partial_t q+\partial_x^2q=2\sigma |q|^2q,
\qquad \sigma\in\{-1,1\}.
\end{equation}
Its conservation laws can be defined by
\begin{equation}\label{eq:wn_recursion}
E_n^\sigma(q) := \Re\int_{\T}q\,\overline{w_n[q]}\,dx,
\end{equation}
where
\begin{equation}\label{eq:wn_definition}
w_1[q]=q, \qquad w_{n+1}[q] = -i\partial_x w_n[q] + \sigma\overline q \sum_{k=1}^{n-1} w_k[q]w_{n-k}[q], \qquad n\geq1.
\end{equation}
We set
\[
\tau=at, \qquad \lambda:=\sqrt{\frac{2|a|}{|b|}}, \qquad \sigma:=-\textnormal{sign}(ab), \qquad u(t,x):=\lambda q(\tau,x),
\]
with $q$ solving \eqref{eq:integrable_nls}. By this normalization, we have that $u$ solves \eqref{NLS}. Accordingly, we define the
conserved quantities for \eqref{NLS} by
\[
E_n(u):= \lambda^2 E_n^\sigma\left(\frac{u}{\lambda}\right).
\]
With this normalization, the leading quadratic part of each odd energy
$E_{2s+1}$ has coefficient one. The remaining coefficients depend on
$a$ and $b$, while the number of factors and the distribution of spatial
derivatives in each monomial are unchanged.}

For any integer $s \ge 1$, the odd invariant $E_{2s+1}$ has the general form
\begin{equation}\label{Es}
E_{2s+1}(u) = \int_{\T} |\partial_x^s u|^2 dx + \int_{\T} P_s(u, \partial_x u, \dots, \partial_x^{s-1} u) dx,
\end{equation}
where $P_s$ is a real-valued polynomial containing lower-order nonlinear terms. Its differential is given by 
\begin{equation}\label{deri}
dE_{2s+1}(u)[h]:=2\Re \langle \nabla E_{2s+1}(u) ,h\rangle_{L^2},\quad \nabla E_{2s+1}(u) = (-1)^s \partial_x^{2s} u + \nabla P_s(u). 
\end{equation}
In particular, it follows from the inductive definition of $E_{2s+1}$ in \cite{kenig2025well} that $P_s(u)$ is a finite linear combination of differential monomials
\begin{equation}\label{mon}
\prod_{i=1}^{M}\partial_x^{\alpha_i}v_i, \qquad v_i\in\{u,\bar u\},
\end{equation}
satisfying
\begin{equation}\label{prop:1}
\max_i\alpha_i\le s-1, \qquad \sum_{i=1}^M\alpha_i\le2s-2
\end{equation}
with $1 \leq M \leq M_s$, for some $M_s>0$ which depends on $s$ and diverges as $s \to \infty$.  In particular, we can be more precise and prove the following characterization.

\begin{lemma}
For any $u \in H^\infty$, $P_s(u)$ is a finite linear combination of elements in
\begin{equation}\label{Dn}
\small
\begin{aligned}
    D_s = \left\{\prod_{i=1}^{M} \partial_x^{\alpha_i} v_i, \, \,  v_i \in \{u, \bar{u}\} \mid  4 \le M \le M_s, \, \,  \max \alpha_i \leq s-1, \, \, \sum_{i=1}^M \alpha_i +M = 2s+2\right\}.
\end{aligned}
\end{equation} 
Moreover, the energies \eqref{Es} satisfy
\begin{equation}\label{equiv}
\left| E_{2s+1}(u)- \|\partial_x^s u\|_{L^2}^2\right|
\leq C(\|u\|_{L^2})
\left(1+\|u\|_{H^s}^{\sigma_s}\right),
\end{equation}
for some $\sigma_s<2$ depending only on $s$, where
$C(\|u\|_{L^2})>0$ depends continuously on the $L^2$-norm of $u$.
\end{lemma}
\begin{proof}
{To prove the first claim and the coercivity estimate, we record two additional structural properties of the monomials \eqref{mon}, recalling that we already have \eqref{prop:1}. For a differential monomial $Q=\prod_{j=1}^M\partial_x^{\alpha_j}v_j$, we consider the weight
\[
\mathrm{wt}(Q):=M+\sum_{j=1}^M\alpha_j.
\]
A direct induction using \eqref{eq:wn_definition} shows that every monomial of $w_n$ has weight
$n$: differentiation increases the weight by one, while $\bar u\,w_kw_{n-k}$ has weight $1+k+(n-k)=n+1$. Moreover, the only linear term of $w_n$ is $(- i \partial_x)^{n-1}u$, while all the remaining terms have at least three factors. Consequently, after multiplication by the additional factor $u$ in the definition of $E_{2s+1}$, since $n=2s+1$ every nonlinear monomial has
\[
M\ge4,\qquad \sum_{j=1}^M\alpha_j+M=2s+2.
\]
}

Fix such a monomial and set
$A:=\sum_{j=1}^{M}\alpha_j$. By H\"older's inequality and the periodic
Gagliardo--Nirenberg inequality,
\[
\begin{aligned}
\left| \int_{\T}\prod_{j=1}^{M}\partial_x^{\alpha_j}v_j\,dx \right| &\leq \prod_{j=1}^{M}\|\partial_x^{\alpha_j}v_j\|_{L^M}\\
&\lesssim \|u\|_{L^2}^{\,M-q_M}\|u\|_{H^s}^{\,q_M},
\end{aligned}
\]
where $v_j\in\{u,\bar u\}$ and
\[
q_M :=\sum_{j=1}^{M} \frac{\alpha_j+\frac12-\frac1M}{s} =\frac{A+\frac{M}{2}-1}{s} =2-\frac{M-2}{2s}.
\]
Here we used $\alpha_j\leq s-1$, so that all the interpolation
exponents lie in $(0,1)$. Since $M\geq4$,
\[
q_M\leq 2-\frac1s<2.
\]
As $P_s$ contains only finitely many monomials, summing the previous
bounds gives
\[
\left|\int_{\T}P_s(u)\,dx\right| \leq C(\|u\|_{L^2}) \left(1+\|u\|_{H^s}^{\,2-\frac1s}\right),
\]
which proves \eqref{equiv}, for instance with $\sigma_s=2-\frac1s$.
\end{proof}

By the coercive estimate above and mass conservation, control of $\left| E_{2s+1} \right| $ yields control of the $H^s$ norm. As stated in \cite{kenig2025well} we can obtain a complex modified Korteweg-de Vries (mKdV) equation, for which this hierarchy of quantities is conserved, by considering the Hamiltonian flow generated by $E_4$,  Namely, by taking the explicit expression of $E_4(u)$ related to \eqref{NLS}, the NLS and mKdV vector fields read as
\begin{align}
X_{NLS}(u) &:= i a \partial_x^2 u + i b |u|^2 u, \label{eq:X_NLS} \\
X_{mKdV}(u) &:=  \partial_x^3 u + \frac{3b}{a} |u|^2 \partial_x u, \label{eq:X_mKdV}
\end{align}
where, with the normalization above, $X_{mKdV}$ coincides, up to reversal of the flow direction and a constant, with the fourth flow of the NLS hierarchy \eqref{eq:X_NLS}. Moreover, we have 
\[
\Re\inner{\nabla E_{2s+1}(u)}{X_{NLS}(u)}_{L^2} = 0, \qquad  \Re\inner{\nabla E_{2s+1}(u)}{X_{mKdV}(u)}_{L^2} = 0,
\]
for all sufficiently smooth functions. We observe that, in the case of $e=0$ and $d= 3bc/a$ the model \eqref{eq:main_pde} is the Hirota equation, with vector field given by the sum of \eqref{eq:X_NLS} and a multiple of \eqref{eq:X_mKdV}. Then in this case $E_{2s+1}$ are integrals of motion of the full equation. Hence, the equation is globally well posed and Theorem \ref{thm:1} follows immediately in this case.
For all other choices of coefficients, to rigorously bound the solution of the perturbative model \eqref{app:D}, we rewrite its spectral approximation $u_N = J_N u_N$ separating the integrable flows from the non-integrable remainder as follows
\begin{equation} \label{eq:approx_Dysthe}
\partial_t u_N = J_N \left( X_{NLS}(u_N) - \eps cX_{mKdV}(u_N) -  \eps R(u_N) \right),
\end{equation}
where the remainder vector field is a first-order vector field, obtained by matching the original coefficients $d$ and $e$ from \eqref{eq:main_pde}
\begin{equation} \label{eq:R_def}
R(u) := \left( d - \frac{3bc}{a} \right) |u|^2 \partial_x u + e u^2 \partial_x \overline{u}.
\end{equation}

\begin{lemma} \label{lem:E5_evolution}
Fix $s \in \mathbb{N}, s\geq 2$. Let $u_N \in C(\mathbb{R}, H^\infty)$ be the solution of the approximate flow \eqref{eq:approx_Dysthe}. Then the NLS invariant $E_{2s+1}$ satisfies the differential inequality
\begin{equation} \label{eq:E5_differential_ineq}
\left| \frac{d}{dt} E_{2s+1}(u_N) \right| \le K\left( \eps + \frac{1}{N^2} \right) \left( 1 + \|u_N\|_{H^s}^{p} \right),
\end{equation}
where $K > 0$ is a constant depending on $s$ and the coefficients of the equation, while the exponent $p > 2$ depends only on $s$.
\end{lemma}

\begin{proof}
We evaluate the time derivative of the invariant along the trajectory of \eqref{eq:approx_Dysthe}. Defining 
\[
X := X_{NLS}(u_N) - \eps c X_{mKdV}(u_N) -  \eps R(u_N),
\]
by \eqref{eq:approx_Dysthe} we have
\begin{equation} \label{eq:Es_chain_rule}
\frac{d}{dt} E_{2s+1}(u_N) =2 \Re \inner{\nabla E_{2s+1}(u_N)}{\partial_t u_N}_{L^2} = 2\Re \inner{\nabla E_{2s+1}(u_N)}{J_N X}_{L^2}.
\end{equation}
We decompose this inner product as
\begin{align} \label{eq:Es_decomp}
\Re \inner{\nabla E_{2s+1}(u_N)}{J_N X}_{L^2} &= \Re \inner{\nabla E_{2s+1}(u_N)}{X}_{L^2} \nonumber \\
&\quad - \Re \inner{\nabla E_{2s+1}(u_N)}{(I - J_N) X}_{L^2}.
\end{align}
By the integrability of the NLS and mKdV flows, their exact invariants are preserved when evaluated on $u_N \in H^\infty$. Therefore, the conservative terms vanish exactly
\begin{equation*}
\Re \inner{\nabla E_{2s+1}(u_N)}{X_{NLS}(u_N) - \eps c X_{mKdV}(u_N)}_{L^2} = 0.
\end{equation*}
Consequently, the first term in \eqref{eq:Es_decomp} simplifies entirely to the remainder
\begin{equation} \label{eq:Es_remainder}
\Re \inner{\nabla E_{2s+1}(u_N)}{X}_{L^2} = -\eps \Re \inner{\nabla E_{2s+1}(u_N)}{R(u_N)}_{L^2}.
\end{equation}

For the second term in \eqref{eq:Es_decomp}, we recall \eqref{deri} and define $A_\eps$ and $g_\eps$ as in \eqref{eq:operators_def}, with $(c,d,e)$ replaced by $(\eps c,\eps d,\eps e)$. Since $u_N=J_Nu_N$ has Fourier support in $|k|\leq N$, the linear term $A_\eps u_N$ has the same property. Using also the self-adjointness of $J_N$, we obtain
\begin{equation} \label{eq:Es_proj_error}
\Re \inner{\nabla E_{2s+1}(u_N)}{(I - J_N) X}_{L^2} = \Re \inner{(I - J_N) \nabla P_s(u_N)}{(I - J_N) g_\eps(u_N)}_{L^2}.
\end{equation}

We start by bounding \eqref{eq:Es_remainder}. We have
\begin{equation}\label{da:stimare}
\eps \Re \inner{\nabla E_{2s+1}(u_N)}{R(u_N)}_{L^2} = \eps \Re \left( \int_{\T} \partial_x^{s} \overline{u}_N \partial_x^s R(u_N) dx + \inner{\nabla P_s(u_N)}{R(u_N)}_{L^2} \right).
\end{equation}
From \eqref{eq:R_def} we know that $R$ is a first-order differential operator, so the most singular terms arise only when all $s$ spatial derivatives fall onto the first derivative of $u_N$. Expanding this via the Leibniz rule, we explicitly isolate the highest-order terms from the commutators
\begin{align}
\partial_x^s R(u_N) &= \partial_x^s \left( \left( d - \frac{3bc}{a} \right) |u_N|^2 \partial_x u_N + e u_N^2 \partial_x \overline{u}_N \right) \nonumber \\
&= \left( d - \frac{3bc}{a} \right) |u_N|^2 \partial_x^{s+1} u_N + e u_N^2 \partial_x^{s+1} \overline{u}_N + Q(u_N), \label{eq:R_expansion}
\end{align}
where $Q(u_N)$ collects all the lower-order derivative products
\begin{equation*}
Q(u_N) := \left( d - \frac{3bc}{a} \right) [\partial_x^s, |u_N|^2] \partial_x u_N + e [\partial_x^s, u_N^2] \partial_x \overline{u}_N.
\end{equation*}
We estimate $Q(u_N)$ in $L^2$ by applying the Leibniz rule, Gagliardo-Nirenberg and H\"older inequalities, to obtain 
\begin{align*}
\norm{[\partial_x^s, |u_N|^2] \partial_x u_N}_{L^2} &\lesssim \norm{\partial_x(|u_N|^2)}_{L^\infty} \norm{\partial_x^{s-1} (\partial_x u_N)}_{L^2} + \norm{\partial_x^s(|u_N|^2)}_{L^2} \norm{\partial_x u_N}_{L^\infty} \\
&\lesssim \norm{u_N}_{W^{1,\infty}}^2 \norm{u_N}_{H^s}.
\end{align*}
By similar estimates for the term involving the coefficient $e$, we have 
\[
\norm{Q(u_N)}_{L^2} \lesssim \norm{u_N}_{W^{1,\infty}}^2 \norm{u_N}_{H^s}.
\]
By Cauchy-Schwarz, its contribution to the integral is bounded by
\begin{equation} \label{eq:Q_bound}
\abs{\eps \Re \int_{\T} (\partial_x^s \overline{u}_N) Q(u_N) dx} \le \eps \norm{\partial_x^s u_N}_{L^2} \norm{Q(u_N)}_{L^2} \lesssim \eps \norm{u_N}_{H^s}^4.
\end{equation}

For the other two terms in \eqref{eq:R_expansion}, an integration by parts prevents any derivative loss. Explicitly, for the first term we compute:
\begin{align}
\Re \int_{\T} (\partial_x^s \overline{u}_N) \left( |u_N|^2 \partial_x^{s+1} u_N \right) dx &= \int_{\T} |u_N|^2 \Re \left( \overline{\partial_x^s u_N} \partial_x (\partial_x^s u_N) \right) dx \nonumber \\
&= \int_{\T} |u_N|^2 \frac{1}{2} \partial_x \left( |\partial_x^s u_N|^2 \right) dx \nonumber \\
&= - \frac{1}{2} \int_{\T} \partial_x(|u_N|^2) |\partial_x^s u_N|^2 dx. \nonumber
\end{align}
Bounding this directly using H\"older's inequality gives
\begin{equation}\label{serve}
\abs{- \frac{1}{2} \int_{\T} \partial_x(|u_N|^2) |\partial_x^s u_N|^2 dx} \le \frac{1}{2} \norm{\partial_x(|u_N|^2)}_{L^\infty} \norm{\partial_x^s u_N}_{L^2}^2 \lesssim \norm{u_N}_{H^s}^4.
\end{equation}
For the second singular term we have
\begin{align*}
\Re \int_{\T} (\partial_x^s \overline{u}_N) \left( u_N^2 \partial_x^{s+1} \overline{u}_N \right) dx &= \Re \int_{\T} u_N^2 \frac{1}{2} \partial_x \left( (\partial_x^s \overline{u}_N)^2 \right) dx \nonumber \\
&= - \frac{1}{2} \Re \int_{\T} \partial_x(u_N^2) (\partial_x^s \overline{u}_N)^2 dx,
\end{align*}
which can be bounded as \eqref{serve}. 

{We need to bound the second term in the r.h.s. of \eqref{da:stimare}. Recall that every monomial of $P_s$ belongs to the class $D_s$ in \eqref{Dn}. Taking the variational derivative of its elements and integrating by parts the derivatives falling on the variation produces monomials of $\nabla P_s$ with $M-1$ factors and with the same total number of derivatives. Hence every monomial of $\nabla P_s$ contains at most $2s-2$ derivatives in total. Since $R$ in \eqref{eq:R_def} is cubic and contains exactly one spatial derivative, every term in
\[
\widetilde R(u_N):= \Re\inner{\nabla P_s(u_N)}{R(u_N)}_{L^2}
\]
is the integral of a monomial with at least six factors and at most $2s-1$ derivatives in total, and then there could be at most one derivative of order $\gamma_1>s$. Moreover, by transferring one derivative at a time from this factor to the remaining ones by integration by parts, no factor carries more than $s$ derivatives. Indeed, after $r=\gamma_1-s$ integrations by parts, the other factors carry at most $s-1$ derivatives in total:
\[
\sum_{j=2}^M\gamma_j+r = \sum_{j=1}^M\gamma_j-s\le s-1\,.
\]
Thus, each of the other factors carries at most $s-1$ derivatives.  Consequently, $\widetilde R(u_N)$ is a finite linear combination of integrals of monomials belonging to
\[
\widetilde D_s= \left\{ \prod_{i=1}^{M}\partial_x^{\alpha_i}v_i,\quad v_i\in\{u,\bar u\}\ \middle|\ 6\leq M\leq M_s+2,\quad \max_i\alpha_i\leq s,\quad \sum_{i=1}^{M}\alpha_i\leq2s-1 \right\}.
\]
After relabeling the factors so that $\alpha_1\geq\alpha_2\geq\cdots \geq \alpha_M$, we have $\alpha_2\leq s-1$; otherwise $\alpha_1 +\alpha_2\geq2s$, contradicting the total derivative count. Hence
\begin{equation}\label{idea}
\left| \int_{\T}\prod_{i=1}^{M}\partial_x^{\alpha_i}v_i\,dx \right| \leq \|\partial_x^{\alpha_1}v_1\|_{L^2} \|\partial_x^{\alpha_2}v_2\|_{L^2} \prod_{i=3}^{M}\|\partial_x^{\alpha_i}v_i\|_{L^\infty} \lesssim \|u_N\|_{H^s}^{p_1},
\end{equation}
by the Sobolev embedding $H^s(\T)\hookrightarrow W^{s-1,\infty}(\T)$, for some $p_1>0$.}

Combining all these terms, we obtain 
\begin{equation}\label{termine:1}
\left|\eps \Re \inner{\nabla E_{2s+1}(u_N)}{R(u_N)}_{L^2}\right| \lesssim \eps (1+\|u_N\|_{H^s}^{p_2}),
\end{equation}
for some $p_2>0$.

\medskip

{Finally, we estimate the projection error in \eqref{eq:Es_proj_error}, by making more precise the structure of $\nabla P_s(u_N)$. As observed above, taking the variational derivative of a monomial in $D_s$ removes one factor and redistributes, by integration by parts, the derivatives of the factor being differentiated among the remaining ones. Therefore, the total number of derivatives is preserved, and every monomial appearing in $\nabla P_s(u_N)$ belongs to
\[
(D_s)_{\nabla}:= \left\{ \prod_{i=1}^{L}\partial_x^{\beta_i}v_i,\quad v_i\in\{u,\bar u\}\ \middle|\ 3\leq L\leq L_s,\quad \max_i\beta_i\leq2s-2,\quad \sum_{i=1}^{L}\beta_i\leq2s-2 \right\}.
\]
Relabeling so that $\beta_1\geq\beta_2\geq\cdots\geq\beta_L$, we have $\beta_2\leq s-1$: otherwise $\beta_1+\beta_2\geq2s$, contradicting $\sum_i\beta_i\leq2s-2$. Hence, all factors except possibly the first one can be estimated in $L^\infty$ by the Sobolev embedding $H^s(\T)\hookrightarrow W^{s-1,\infty}(\T)$.}

On the other hand, since $u_N=J_Nu_N$ and $\beta_1\le 2s-2$, the frequency localization gives
\[
\|\partial_x^{\beta_1}u_N\|_{L^2} \lesssim N^{\max\{\beta_1-s,0\}}\|u_N\|_{H^s} \lesssim N^{s-2}\|u_N\|_{H^s}.
\]
Hence, for every monomial occurring in $\nabla P_s(u_N)$,
\[
\left\| \prod_{i=1}^{L}\partial_x^{\beta_i}v_i \right\|_{L^2} \lesssim N^{s-2}\|u_N\|_{H^s}^{q},
\]
for some positive $q$. Since $\nabla P_s$ contains only finitely many such monomials, there exists an exponent $p_3>0$, depending only on $s$, such that
\begin{equation}\label{eq:gradPs_projection_bound}
\|\nabla P_s(u_N)\|_{L^2} \lesssim N^{s-2}\left(1+\|u_N\|_{H^s}^{p_3}\right).
\end{equation}

We now estimate the second factor in \eqref{eq:Es_proj_error}. 
We decompose the scaled nonlinearity as $g_\eps(u)=g_0(u)+\eps g_1(u),$
where
\[
g_0(u):=ib|u|^2u, \qquad g_1(u):=-d|u|^2\partial_xu-eu^2\partial_x\overline{u}.
\]
By property 4 of Lemma \ref{lem:JN_properties},
\begin{align}
\|(I-J_N)g_\eps(u_N)\|_{L^2} &\le \|(I-J_N)g_0(u_N)\|_{L^2} +\eps\|(I-J_N)g_1(u_N)\|_{L^2} \nonumber\\
&\lesssim \frac{1}{N^s}\|g_0(u_N)\|_{H^s} + \frac{\eps}{N^{s-1}}\|g_1(u_N)\|_{H^{s-1}}. \label{eq:geps_projection_bound}
\end{align}
Since $s\ge2$, the standard Sobolev product estimates give
\[
\|g_0(u_N)\|_{H^s} \lesssim \|u_N\|_{H^s}^3, \qquad \|g_1(u_N)\|_{H^{s-1}} \lesssim \|u_N\|_{H^s}^3.
\]
Therefore,
\begin{equation}\label{eq:geps_projection_bound_final}
\|(I-J_N)g_\eps(u_N)\|_{L^2} \lesssim \left( \frac{1}{N^s} + \frac{\eps}{N^{s-1}} \right) \|u_N\|_{H^s}^3.
\end{equation}

Combining \eqref{eq:Es_proj_error}, \eqref{eq:gradPs_projection_bound}, and
\eqref{eq:geps_projection_bound_final}, we obtain
\begin{align}
\left| \Re\inner{(I-J_N)\nabla P_s(u_N)} {(I-J_N)g_\eps(u_N)}_{L^2} \right| &\le \|(I-J_N)\nabla P_s(u_N)\|_{L^2} \|(I-J_N)g_\eps(u_N)\|_{L^2} \nonumber\\
&\le \|\nabla P_s(u_N)\|_{L^2} \|(I-J_N)g_\eps(u_N)\|_{L^2} \nonumber\\
&\lesssim N^{s-2} \left( \frac{1}{N^s} + \frac{\eps}{N^{s-1}} \right) \left(1+\|u_N\|_{H^s}^{p_3+3}\right) \nonumber\\
&\lesssim \left( \frac{1}{N^2} + \eps \right) \left(1+\|u_N\|_{H^s}^{p_3+3}\right), \label{eq:Es_proj_explicit}
\end{align}
where in the last inequality we used $N\ge1$.

Combining \eqref{eq:Es_proj_explicit} with \eqref{termine:1}, we conclude that
\[
\left| \frac{d}{dt}E_{2s+1}(u_N) \right| \lesssim \left( \eps+\frac{1}{N^2} \right) \left( 1+\|u_N\|_{H^s}^{p} \right),
\]
for some $p>2$ depending only on $s$, which proves
\eqref{eq:E5_differential_ineq}.
\end{proof}

\begin{lemma} \label{lem:timescale_existence}
Let $s \geq 2 $ and $u_0 \in H^s$. There exists a positive constant $\tilde{K}_0$, such that for any choice of $\tilde{K}_1 > \tilde{K}_0$, we can find $\tilde{K}_2 > 0$ for which the following holds.

For every $\eps \in (0, 1)$, there exists $N_*=N_*(\eps)$ such that for all $N > N_*$, the solution $u_N$ to the approximate equation \eqref{eq:approx_Dysthe} with initial data $J_N u_0$ satisfies
\begin{equation}\label{claim:Nstar}
\sup_{t \in [-\tilde{T}, \tilde{T}]} \|u_N(t)\|_{H^{[s]}} \le \tilde{K}_1, 
\end{equation}
where the symmetric existence time scale is $\tilde{T} = \tilde{K}_2 \eps^{-1}$ and $[s]=\max\{k \in \mathbb{N} \mid k \leq s\}$.
\end{lemma}

\begin{proof}
Without loss in regularity we consider $s \in \mathbb{N}_{\geq 2}$ and $u_{0} \in H^{s}$, since for a general $s \geq 2$ the following estimates hold true for $[s] \in \mathbb{N}_{\geq 2}$. Let $u_{N} \in C(\mathbb{R}, H^{\infty})$ be the solution of the approximate problem \eqref{eq:approx_Dysthe} with scaled coefficients.

Under these scaling hypotheses, there exists $N_*$ such that for any $N \geq N_*(\eps)$, we have $N > \eps^{-1/2}$. Then, by \eqref{eq:E5_differential_ineq}, we have
\begin{equation} \label{eq:E2s1_eps_bound}
\left|\frac{d}{dt}E_{2s+1}(u_{N})\right| \leq \eps C_{0} \left(1 + \|u_{N}\|_{H^{s}}^{p}\right),
\end{equation}
where $C_{0} > 0$ is a constant uniform in $N$ and $\eps$, and where $p > 2$ depends only on $s$.

Moreover, by \eqref{equiv} we have
\begin{align} \label{eq:E2s1_initial_bound}
\left| E_{2s+1}(u_{N}(0)) \right| \lesssim \left(1 + \|u_{0}\|_{H^{s}}^{2}\right),
\end{align}
where we have used $\sigma_s<2$ and the uniform boundedness of the Fourier truncation at time $t=0$, with the omitted constants uniform in $N$ and $\eps$. 
By integrating \eqref{eq:E2s1_eps_bound} over time for any $t \geq 0$, we obtain
\begin{equation}\label{1:imp}
\left| E_{2s+1}(u_{N}(t))- E_{2s+1}(u_{N}(0)) \right|
\lesssim
 \eps \int_{0}^{t}
\left(1+\|u_{N}(\tau)\|_{H^{s}}^{p}\right)d\tau.
\end{equation}

The following computation follows the same argument as in the proof of Corollary \ref{e0GWP} (see \eqref{eq:H1_reconstruction} onwards). For the sake of completeness, we recall only the general idea.

{Combining \eqref{1:imp}, \eqref{equiv}, the conservation of the $L^2$ norm, and Young's inequality, there exist constants $C_2,C_3>0$, uniform in $N$ and $\eps$, such that
\begin{equation}\label{eq:longtime_scalar_ineq}
X(t) \leq C_2 + C_3\eps\int_0^t X(\tau)^q\,d\tau =:G(t), \qquad X(t):=1+\|u_N(t)\|_{H^s}^2, \qquad q:=\frac{p}{2}>1.
\end{equation}
Here we used $\sigma_s<2$ to absorb the lower-order term in \eqref{equiv} by Young's inequality. Since $X(t)\leq G(t)$, we have
\[
G(t) \leq \frac{C_2} {\left( 1-(q-1)C_4\eps t \right)^{1/(q-1)}}, \qquad C_4:=C_3C_2^{q-1},
\]
as long as the denominator is positive. By choosing $\tilde K_0>0$ such that $C_2<1+\tilde K_0^2$, for any $\tilde K_1 > \tilde K_0$,  there exists $\tilde K_2>0$ uniform in $\eps$ such that
\[
\|u_N(t)\|_{H^s}\leq\tilde K_1 \qquad \text{for } 0\leq t\leq\tilde K_2\eps^{-1}.
\]
The same estimate holds for negative times by \eqref{eq:E2s1_eps_bound}.}
\end{proof}

By applying Lemma \ref{lem:timescale_existence} we can now prove our second main Theorem \ref{thm:1}.
\begin{proof}[Proof of Theorem \ref{thm:1}]
We start by considering $s \in \mathbb{N}_{\geq 2}$ and $u_{0} \in H^{s}$. In this case, \eqref{claim:Nstar} is satisfied for $s$. The uniform bound $\eqref{claim:Nstar}$ allows us to apply the compactness/convergence argument of Lemma \ref{UNintermediate2} on $[-\widetilde T,\widetilde T]$. Hence the approximate solutions converge to the unique exact solution $u$ on this interval, weakly in $H^s$ and strongly in $H^{s'}$, $s'<s$. Exploiting the standard property of the weak lower semicontinuity of norms, we have
\begin{equation*}
\|u(t)\|_{H^{s}} \leq \liminf_{N\to\infty}\|u_{N}(t)\|_{H^{s}} \leq \sup_{N\geq N_*}\|u_{N}(t)\|_{H^{s}} \leq \widetilde{K}_{1},
\end{equation*}
for any $t \in [-\tilde T, \tilde T]$, proving \eqref{stability} in the integer case. This ensures that the exact solution does not blow up in $H^{s}$ on a time interval of length proportional to $\eps^{-1}$, with a bound uniform in $\eps$, for any $s \in \mathbb{N}_{\geq 2}$.  

{Since the solution has been constructed on $[-\tilde T,\tilde T]$, the local theory can be restarted and then its maximal existence times satisfy
\[
T_{1,\textnormal{max}},T_{2,\textnormal{max}} > \tilde T=\tilde K_2\eps^{-1},
\]
which proves \eqref{stima:longtime} when $s$ is an integer. It remains to consider $s\geq2$ with $s\notin\mathbb N$. Set $m:=[s]$. Since $u_0\in H^s\subset H^m$, we can apply the integer case already proved at regularity $m$. Hence, for the same fixed initial datum, there exists a solution $v\in C([-\tilde T,\tilde T],H^m)$ to \eqref{app:D}, where $\tilde T=\tilde K_2\eps^{-1}$, such that
\begin{equation}\label{eq:integer_stability_fractional}
\sup_{t\in[-\tilde T,\tilde T]} \|v(t)\|_{H^m} \leq\tilde K_1.
\end{equation}
Now let $u\in C((-T_{1,\textnormal{max}},T_{2,\textnormal{max}}),H^s)$ be the maximal $H^s$ solution with initial datum $u_0$. On the common
interval of existence, both $u$ and $v$ are $H^m$ solutions of \eqref{app:D} with the same initial datum. Therefore, by the uniqueness part of Theorem \ref{GWPG}, $u(t)=v(t)$ throughout their common interval of existence. Assume by contradiction that $T_{2,\textnormal{max}}\leq\tilde T$. Then, by
\eqref{eq:integer_stability_fractional},
\[
\sup_{0\leq t<T_{2,\textnormal{max}}} \|u(t)\|_{H^1} \leq \sup_{0\leq t<T_{2,\textnormal{max}}} \|v(t)\|_{H^m} \leq\tilde K_1,\]
which contradicts the continuation criterion of Theorem \ref{GWPG}. Consequently, $T_{2,\textnormal{max}}>\tilde T$, and the same argument in the negative time direction gives $T_{1,\textnormal{max}}>\tilde T,$ which proves \eqref{stima:longtime} for arbitrary real $s\geq2$. Finally, since $[-\tilde T,\tilde T]$ is contained in the maximal existence interval of $u$, the uniqueness statement of Theorem \ref{GWPG} gives $u(t)=v(t)$ for any $t \in[-\tilde T,\tilde T]$. Combining this identity with
\eqref{eq:integer_stability_fractional}, and recalling that $m=[s]$, we obtain \eqref{stability}, which concludes the proof.}
\end{proof}

\end{document}